\documentclass[a4paper, 12pt, oneside, notitlepage]{amsart}
\usepackage{amsmath,amssymb,amsthm,graphicx,mathrsfs,bbm,url}
\usepackage{amsthm}
\usepackage{wrapfig}
\usepackage{enumitem}
\usepackage{mathtools}
\usepackage[all]{xy}
\usepackage[utf8]{inputenc}
\usepackage[usenames,dvipsnames]{color}
\usepackage[colorlinks=true,linkcolor=Red,citecolor=Green]{hyperref}
\usepackage[super]{nth}
\usepackage[open, openlevel=2, depth=3, atend]{bookmark}
\hypersetup{pdfstartview=XYZ}
\usepackage[font=footnotesize]{caption}
\usepackage{a4wide}
\usepackage{subcaption}
\usepackage{caption}
\usepackage{tikz}
\usepackage{setspace}
\usepackage[normalem]{ ulem }
\usepackage{soul}
\usetikzlibrary{decorations.pathreplacing}
\usetikzlibrary{arrows}
\usetikzlibrary{shapes.misc}
\usetikzlibrary{shapes.symbols}
\usetikzlibrary{patterns}
\usepackage{epstopdf}
 
\usepackage{hyperref}

\usepackage{listings}
\usepackage{color}

\definecolor{dkgreen}{rgb}{0,0.6,0}
\definecolor{gray}{rgb}{0.5,0.5,0.5}
\definecolor{mauve}{rgb}{0.58,0,0.82}

\theoremstyle{plain}
\newtheorem{theorem}{Theorem}[section]
\newtheorem*{theorem*}{Theorem}

\newtheorem{lemma}[theorem]{Lemma}
\newtheorem{proposition}[theorem]{Proposition}

\newtheorem{conjecture}[theorem]{Conjecture}

\theoremstyle{definition}
\newtheorem{definition}[theorem]{Definition}

\theoremstyle{remark}
\newtheorem{remark}[theorem]{Remark}

\numberwithin{equation}{section}

\newcommand{\R}{\mathbb{R}}

\newcommand{\Z}{\mathbb{Z}}

\newcommand{\M}{\mathcal{M}}
\newcommand{\N}{\mathbb{N}}

\newcommand{\V}{\mathbb{V}}
\newcommand{\X}{\mathbf{X}}

\newcommand{\HH}{\mathbb{H}}
\newcommand{\Ss}{\mathbb{S}}
\newcommand{\eps}{\varepsilon}

\newcommand{\mc}{\mathcal}

\newcommand{\Sym}{\mathrm{Sym}}

\DeclareMathOperator{\Tr}{Tr}

\DeclareMathOperator{\dd}{D}

\DeclareMathOperator{\vol}{vol}

\DeclareMathOperator{\id}{Id}
\DeclareMathOperator{\E}{\mathcal{E}}

\DeclareMathOperator{\e}{\mathbf{e}}

\DeclareMathOperator{\End}{\mathrm{End}}

\newcommand{\be}{\begin{equation}}
\newcommand{\ee}{\end{equation}}

\DeclarePairedDelimiter\floor{\lfloor}{\rfloor}

\def\beq{\begin{equation}}
\def\eeq{\end{equation}}

\def\Oo{\mathrm{O}}
\def\SO{\mathrm{SO}}

\def\beq{\begin{equation}}
\def\eeq{\end{equation}}
\def\bea{\begin{eqnarray*}}
\def\eea{\end{eqnarray*}}

\title
[On the ergodicity of the frame flow on even-dimensional manifolds]
{On the ergodicity of the frame flow on even-dimensional manifolds}

\author{Mihajlo Ceki\'c}
\address{Institut f\"ur Mathematik, Universit\"at Z\"urich, Winterthurerstrasse 190, CH-8057 Z\"urich, Switzerland}
\email{mihajlo.cekic@math.uzh.ch}

\author{Thibault Lefeuvre}
\address{Université de Paris and Sorbonne Université, CNRS, IMJ-PRG, F-75006 Paris, France.}
\email{tlefeuvre@imj-prg.fr}

\author{Andrei Moroianu}
\address{Université Paris-Saclay, CNRS,  Laboratoire de mathématiques d'Orsay, 91405, Orsay, France
and Institute of Mathematics “Simion Stoilow” of the Romanian Academy, 21 Calea Grivitei, 010702 Bucharest, Romania}
\email{andrei.moroianu@math.cnrs.fr}

\author{Uwe Semmelmann}
\address{Institut f\"ur Geometrie und Topologie, Fachbereich Mathematik, Universit{\"a}t Stuttgart, Pfaffenwaldring 57, 70569 Stuttgart, Germany
}
\email{uwe.semmelmann@mathematik.uni-stuttgart.de}

\begin{document}

\begin{abstract} It is known that the frame flow on a closed $n$-dimensional Riemannian manifold with negative sectional curvature is ergodic if $n$ is odd and $n \neq 7$. In this paper we study its ergodicity in the remaining cases. For $n$ even and $n \neq 8, 134$, we show that:
\begin{enumerate}
\item if $n \equiv 2$ mod $4$ or $n=4$, the frame flow is ergodic if the manifold is $\sim 0.3$-pinched,
\item if $n \equiv 0$ mod $4$, it is ergodic if the manifold is $\sim 0.6$-pinched.
\end{enumerate}
In the three dimensions $n=7,8,134$, the respective pinching bounds that we need in order to prove ergodicity are $0.4962...$, $0.6212...$, and $0.5788...$. This is a significant improvement over the previously known results and a step forward towards solving a long-standing conjecture of Brin asserting that $0.25$-pinched even-dimensional manifolds have an ergodic frame flow.
\end{abstract}
\maketitle

\section{Introduction}

Let $(M^n,g)$ be a smooth closed (compact, without boundary) oriented Riemannian manifold with negative sectional curvature of dimension $n \geq 3$. Let $SM \to M$ be the unit tangent bundle and let $FM \to M$ be the principal $\mathrm{SO}(n)$-bundle of \emph{oriented orthonormal bases} over $M$. A point $w \in FM$ over $x \in M$ is the data of an oriented orthonormal basis $(v, \e_2, \ldots, \e_n)$ of $(T_xM, g_x)$. Equivalently, we will see $FM$ as a principal $\mathrm{SO}(n-1)$-bundle over $SM$ by the projection map $p : FM \to SM$ defined as $p(x, v,\e_2,\ldots,\e_n) = (x, v)$. 

We denote by $(\varphi_t)_{t \in \R}$ the geodesic flow on $SM$ and by $(\Phi_t)_{t \in \R}$ the frame flow on $FM$, defined in the following way: given $t \in \R,w \in FM$, the point $\Phi_t(w)$ is obtained by flowing $(x,v)$ by the geodesic flow and parallel transport along this geodesic of the remaining vectors $(\e_2,\ldots,\e_n)$. They satisfy the obvious commutation relation $p \circ \Phi_t = \varphi_t \circ p$, that is the frame flow is an extension of the geodesic flow. When $(M,g)$ has negative sectional curvature (or more generally, when the geodesic flow is \emph{Anosov}, i.e. uniformly hyperbolic), the frame flow is a typical example of a \emph{partially hyperbolic flow}, see \cite{Hasselblatt-Pesin-06}. Since it preserves a natural smooth measure (the product measure of the Liouville measure on the unit tangent bundle and the Haar measure on the group), one of the main questions from the perspective of dynamical systems is to understand its ergodicity with respect to that measure. In negative curvature, it is known that the frame flow is ergodic when $M$ is odd-dimensional \cite{Brin-Gromov-80} and $n \neq 7$, without any further restriction on the metric. However, the situation is more complicated in even dimensions and for $n=7$. 

We will say that the negatively-curved manifold $(M,g)$ has \emph{$\delta$-pinched curvature} for some $\delta \in (0,1]$ if there exists a constant $K > 0$ such that the sectional curvature $\kappa$ satisfies the uniform bounds:
\begin{equation}
\label{equation:pinching}
-K \leq \kappa(u \wedge v) \leq -\delta K,
\end{equation}
for any two-plane $u \wedge v$ in $TM$. We will say that it has \emph{strictly $\delta$-pinched curvature} if the inequality on the right of \eqref{equation:pinching} is strict. Note that in even dimensions, Kähler manifolds cannot have an ergodic frame flow and these are at most $0.25$-pinched \cite{Berger-60-1}. Brin thus formulated the natural conjecture (see \cite[Conjecture 2.6]{Brin-82}):

\begin{conjecture}[Brin '82]
\label{conjecture:brin}
If $(M,g)$ is strictly $0.25$-pinched, then the frame flow is ergodic.
\end{conjecture}
More precisely, Brin conjectures that with the same assumption, the frame flow is \emph{Bernoulli} which implies ergodicity (see \cite{Dolgopyat-Kanigowski-Rodriguez-Hertz-24} for the definition of Bernoulli and an up to date survey of this property). Brin also conjectures in the same paper that the frame flow should be ergodic and Bernoulli as long as the holonomy group of the manifold is $\mathrm{SO}(n)$, see \cite[Conjecture 2.9]{Brin-82}. It is also reasonable to expect that the frame flow is ergodic in dimension $7$, without any pinching condition. So far, positive answers to the conjectural ergodicity of the frame flow in even dimensions and dimension $7$ were obtained for a pinching $\delta$ close to $1$: strictly greater than $0.8649$ in even dimensions different from $8$ \cite{Brin-Karcher-83}, and strictly greater than $0.9805...$ in dimensions $7$ and $8$ \cite{Burns-Pollicott-03}\footnote{Note that our convention is different from \cite{Brin-Karcher-83,Burns-Pollicott-03} as our pinching is the square of their pinching.}. Ergodicity of the frame flow also holds on an open and dense set of $C^3$-metrics with negative curvature \cite[Section 5]{Brin-82}. However, there has been no progress on Conjecture \ref{conjecture:brin} in the past twenty years. \\

In this paper, we study the dimensions not covered by \cite{Brin-Gromov-80} and prove the following:

\begin{theorem}
\label{theorem:ergodicity}
Let $(M^n,g)$ be a closed $n$-dimensional negatively curved oriented Riemannian manifold with $\delta$-pinched curvature and $n \geq 3$. In the cases where $n$ is even or $n = 7$, the frame flow is ergodic if $\delta > \delta(n)$, where $\delta(n)$ is given by
\small
\begin{equation}
\label{equation:threshold-final}
\begin{array}{cl}
0.2928..., &\text{ if }\, n=4, \\
0.2823..., & \hspace{0.57cm} n=6,  \\
0.4962..., & \hspace{0.57cm} n=7, \\
0.6212..., & \hspace{0.57cm} n=8, \\
0.5788..., & \hspace{0.57cm} n=134, \\
 & \\
\tfrac{\tfrac{2}{3} \sqrt{3(n^2 - 1)} + \tfrac{1}{2}}{3(n + 1) + \tfrac{2}{3}\sqrt{3(n^2 - 1)} - \tfrac{1}{2}},  & \text{ if }\, n \geq 10,\, n \neq 134,\, n \equiv 2 \text{ mod } 4, \\
& \\
\tfrac{n+5+\tfrac{8}{3}\sqrt{(n-1)(n+2)}+\tfrac{2(n+2)(n+4)}{3(n+1)(n+6)}\left( n+3+\tfrac{4}{3} \sqrt{3(n^2-1)}\right)}{3(n+1) + \tfrac{8}{3}\sqrt{(n-1)(n+2)} + \tfrac{2(n+2)(n+4)}{3(n+1)(n+6)} \left(5n+3+\tfrac{4}{3}\sqrt{3(n^2-1)} \right)}, & \text{ if}~ n\geq 12,\, n \equiv 0 \text{ mod } 4.
\end{array}
\end{equation}
\normalsize
Asymptotically, $\delta(4\ell+2) \to_{\ell \to \infty} 0.2779...$ and $\delta(4\ell) \to_{\ell \to \infty} 0.5572...$. Moreover, the sequence $(\delta(4\ell+2))_{\ell \geq 2}$ is increasing and $\delta(10)= 0.2725...$, while $(\delta(4\ell))_{\ell \geq 3}$ is decreasing and $\delta(12) = 0.5948...$.
\end{theorem}

Theorem \ref{theorem:ergodicity} is illustrated by Figure \ref{figure}.

\begin{center}
\begin{figure}[htbp!]
\includegraphics[scale=0.6]{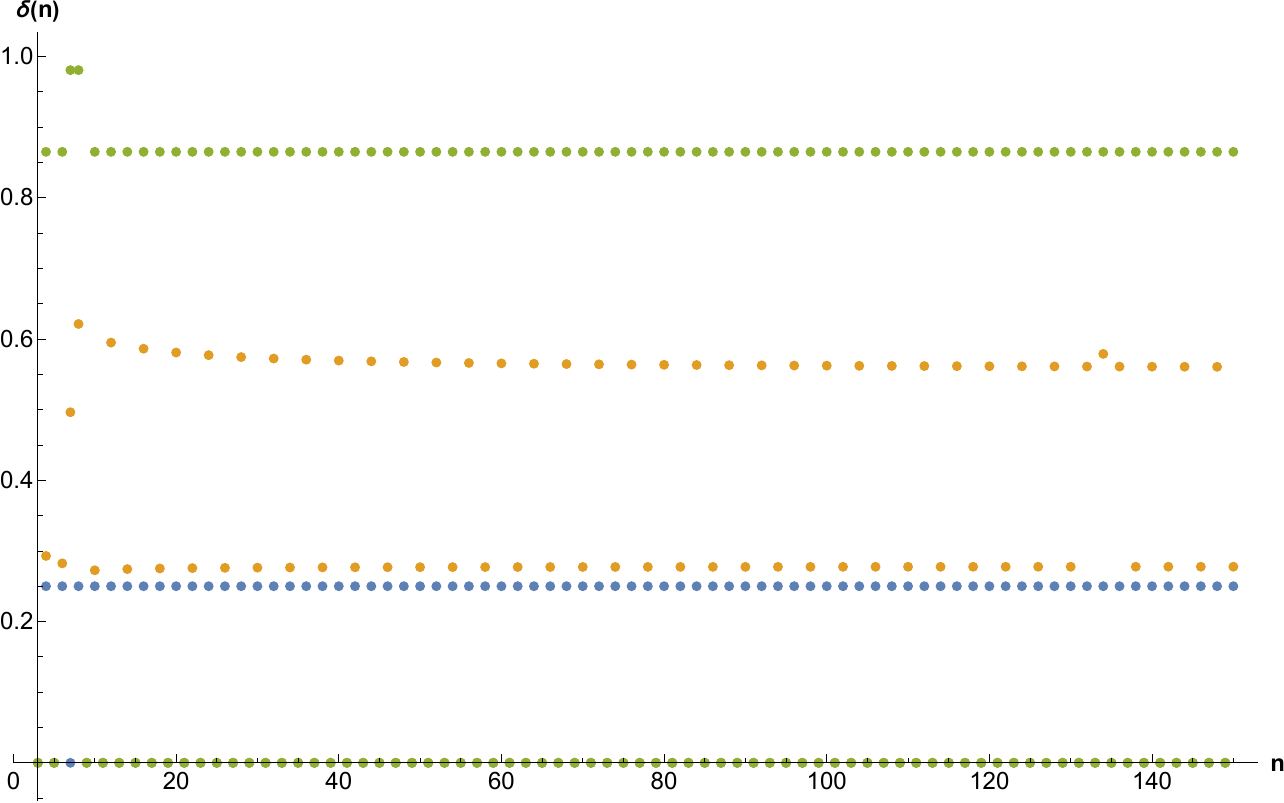}
\caption{In green: the bounds existing in the literature \cite{Brin-Gromov-80,Brin-Karcher-83,Burns-Pollicott-03}. In orange: the bounds provided by Theorem \ref{theorem:ergodicity}. In blue: the conjectural $0.25$ threshold.}
\label{figure}
\end{figure}
\end{center} 

The strategy of the proof is the following. When the frame flow is not ergodic, one can define a strict subgroup $H$ of $\mathrm{SO}(n-1)$ (defined up to conjugation) called the \emph{transitivity group}, and an $H$-subbundle of $FM$ on which the flow is ergodic, see \cite{Brin-75-1,Brin-75-2}, \S\ref{ssection:dynamics1}, and \S\ref{ssection:dynamics2} for further details. This subgroup gives in particular a reduction of the structure group of the frame bundle over the sphere $\Ss^{n-1}$. Using topological arguments, one can exclude most subgroups of $\mathrm{SO}(n-1)$ and only a few cases survive, see \S\ref{ssection:topology}. We show by representation theory and the non-Abelian Liv\v sic theory developed in \cite{Cekic-Lefeuvre-21-1}, that for each possible subgroup $H$ one can construct a flow-invariant section on the unit tangent bundle $SM$ which takes values either in $p$-forms (for $p=1,2,3$) or symmetric endomorphisms of the \emph{normal bundle} (the tangent bundle to the spherical fibers), see Theorem \ref{theorem:reduction}. In turn, the existence of such an invariant section gives rise on the base manifold $M$ to a new object, which we call \emph{normal twisted conformal Killing tensor}: it is a symmetric tensor twisted by some vector bundle, satisfying an algebraic constraint (\emph{normal} condition) and a differential equation similar to the conformal Killing equation (see \cite{Dairbekov-Sharafutdinov-10, Heil-Moroianu-Semmelmann-16, Guillarmou-Paternain-Salo-Uhlmann-16} for further details on the conformal Killing equation). Under some pinching condition and by curvature estimates, we can then rule out the existence of such a non-trivial object using the twisted Pestov identity, see Theorem \ref{theorem:invariant-structures} and \S\ref{section:invariant}. In comparison, earlier results on the ergodicity of the frame flow are based on purely topological arguments \cite{Brin-Gromov-80} (the result in \cite{Brin-Gromov-80} is also naturally re-established in the course of our proof, with essentially the same argument) or on geometric arguments on the universal cover of the manifold, see \cite{Brin-Karcher-83, Burns-Pollicott-03}. 

It can be checked that ergodicity implies mixing for frame flows, see \cite{Lefeuvre-21}. Moreover, it was shown very recently that ergodicity implies rapid mixing \cite{Cekic-Lefeuvre-24, Pollicott-Zhang-24}, that is mixing faster than any polynomial in time. While (negatively curved) K\"ahler manifolds do not have ergodic frame flow, we may ask whether their \emph{unitary frame flow} (that is, the frame flow restricted to the unitary frame bundle) is ergodic. In this setting, \cite{Brin-Gromov-80} show ergodicity if $m := \frac{n}{2}$ is odd or $m = 2$, while in \cite{Cekic-Lefeuvre-Moroianu-Semmelmann-24} we show ergodicity for $m$ even and $m \neq 4, 28$, under a suitable assumption of \emph{holomorphic} pinching. Ergodicity of the frame flow of an arbitrary real vector bundle $(\E, \nabla) \to M$ equipped with a fibrewise inner product and a compatible (orthogonal) connection $\nabla$ was considered in \cite{Cekic-Lefeuvre-22}, where it was established that if the rank $r$ of $\E$ satisfies $r \leq \sqrt{n}$ and the holonomy group of $(\E, \nabla)$ is $\mathrm{SO}(r)$, then the frame flow is ergodic; this provides some evidence in support of \cite[Conjecture 2.9]{Brin-82}.

We believe that the new approach developed in the present paper should eventually lead to a proof of Conjecture \ref{conjecture:brin}, at least in dimensions $4$ and $4\ell +2, \ell > 0$. At this stage, it is not clear whether we use the full strength of the twisted Pestov identity or if some improvements could be achieved in the computations. In particular, numerical experiments could help understand how sharp the inequalities derived from the Pestov identity in \S\ref{section:invariant} are. Moreover, once the normal twisted conformal Killing tensor is obtained in \S\ref{section:invariant}, there might also be an alternative approach to the Pestov identity (e.g. another energy identity on the unit tangent bundle) in order to conclude that this tensor is zero. More generally, we believe that the approach of the present paper should allow one to study ergodicity of general isometric extensions (to a compact fiber bundle) of the geodesic flow over a negatively curved Riemannian manifold, see \cite{Lefeuvre-21} where this is further discussed. \\

\noindent {\bf Organisation of the article.} In \S \ref{section:preliminary} we recall some properties of curvature under the pinching assumption, Fourier analysis and symmetric tensors, geometry of $SM$, and we state the Pestov identity. In \S \ref{section:dynamics} we introduce the transitivity group for principal bundle extensions of Anosov flows, state a correspondence between flow-invariant sections of (associated) vector bundles and fixed vectors of (associated) representations of the transitivity group, and show that non-ergodicity of the frame flow implies the existence of non-zero flow-invariant structures on $SM$. Finally, in \S \ref{section:invariant} we prove delicate bounds involving the pinching constant for the terms appearing in the Pestov identity, which in turn enables us to complete the proof. Note that a survey of the results in this paper can be found in \cite{Cekic-Lefeuvre-Moroianu-Semmelmann-22}.
\\

\noindent \textbf{Acknowledgements:} M.C. has received funding from the European Research Council (ERC) under the European Union’s Horizon 2020 research and innovation programme (grant agreement No. 725967), and an Ambizione grant (project number 201806) from the Swiss National Science Foundation. A.M. was partly supported by the PNRR-III-C9-2023-I8 grant CF 149/31.07.2023 {\em Conformal Aspects of Geometry and Dynamics}. We thank Julien Marché for fruitful discussions, and the anonymous referees for their useful comments and remarks, which allowed us to significantly improve the exposition.

\section{Preliminaries}\label{section:preliminary}

In this section, we provide technical preliminaries necessary throughout this article.

\subsection{Bounds on the curvature tensor}\label{ssection:curvature-tensor-bounds}

Let $(M, g)$ be a smooth Riemannian manifold. We define the curvature tensor $R \in C^\infty(M, \Lambda^2T^*M \otimes \End(TM))$ as:
\[R(X, Y)Z = \nabla^2Z(X, Y) = \nabla_X \nabla_Y Z - \nabla_Y \nabla_X Z - \nabla_{[X, Y]}Z,\]
where $\nabla$ is the Levi-Civita connection, and $X, Y, Z$ are vector fields on $M$. We note that the notation $\nabla^2$ means that we apply $\nabla$ twice, where the second $\nabla$ is defined using the Leibniz rule as a map $C^\infty(M, T^*M \otimes TM) \to C^\infty(M, \Lambda^2 T^*M \otimes TM)$ (this is a classical definition in Riemannian geometry). Throughout the paper we will identify 1-forms with tangent vectors, and more generally $\Lambda^pTM$ with $\Lambda^pT^*M$ via the metric  $\langle{\cdot, \cdot}\rangle = g(\cdot, \cdot)$ and we view the curvature tensor as a $4$-tensor by defining $R(X, Y, Z, T) := \langle{R(X, Y)Z, T}\rangle$. We denote by $\kappa(a \wedge b) := \tfrac{R(a, b, b, a)}{|a|^2|b|^2 - \langle{a, b}\rangle^2}$ the sectional curvature of the plane spanned by two tangent vectors $a$ and $b$. The curvature tensor of a space of constant sectional curvature $-1$ is given by
\begin{equation}\label{eq:Gdef}
	G(a, b)c = \langle{a, c}\rangle.b - \langle{b, c}\rangle.a, \quad G(a, b, c, d) = \langle{a, c}\rangle.\langle{b, d}\rangle - \langle{b, c}\rangle.\langle{a, d}\rangle, 
\end{equation}
for any tangent vectors $a, b, c, d$. If $(M,g)$ has $\delta$-pinched negative sectional curvature (for some $0 < \delta \leq 1$), that is, the sectional curvature satisfies $-1 \leq \kappa(a \wedge b) \leq - \delta$ for all $2$-planes $a \wedge b$, we set:
\begin{equation}\label{eq:R_0def}
	R_0 := R - \tfrac{1 + \delta}{2} G.
\end{equation}
Observe that for unit tangent vectors $a, b$, $R_0$ ``centers'' the tensor $R$ around zero:
\begin{equation}\label{eq:curvatureauxiliary}
	|R_0(a, b, b, a)| \leq \tfrac{1 - \delta}{2} (1 - \langle{a, b}\rangle^2) \leq \tfrac{1 - \delta}{2}.
\end{equation}
Using polarisation identities for the curvature, \cite[Lemma 3.7]{Bourguignon-Karcher-78} shows that:

\begin{lemma}\label{lemma:bk}
For all unit vectors $a, b, c, d \in T_xM$, the following estimate holds:
\begin{equation}
\label{equation:bk-pas-sharp}
	|R_0(a, b, c, d)| \leq \tfrac{2 ( 1 - \delta)}{3}.
\end{equation}
This estimate is sharp for the complex hyperbolic space, with $\delta=0.25$.
\end{lemma}

For $p \in \left\{1,\ldots,n\right\}$, the connection on $\Lambda^p TM \to M$ is induced by the Levi-Civita connection by asking the Leibniz rule to hold. The induced curvature $R^{\Lambda^p}$ on the bundle $\Lambda^p TM$ (for $p=1,\ldots,n$) is given by:
\begin{equation}
\label{equation:algebra}
\begin{split}
R^{\Lambda^p}(a,b) \big(e_1 \wedge \ldots \wedge e_p\big) = & (R(a,b) e_1) \wedge e_2 \wedge \ldots \wedge e_p  \\
& \hspace{2cm} + \ldots + e_1 \wedge \ldots \wedge e_{p-1} \wedge  (R(a,b) e_p)
\end{split}
\end{equation}
where $x \in M, a,b, e_1, \ldots, e_p \in T_xM$; the constant curvature map $G^{\Lambda^p}$ is similarly defined from $G$. The scalar product on decomposable elements in $\Lambda^p TM$ is given by the determinant, namely
\[
\langle \eta_1 \wedge \ldots \wedge \eta_p, \omega_1 \wedge \ldots \wedge \omega_p \rangle = \det (\langle \eta_i, \omega_j \rangle)_{1 \leq i,j \leq p}.
\]
The induced connection on $\Lambda^pTM$ is compatible with this scalar product, that is, it is \emph{orthogonal} (also called metric). As before, the curvature $R^{\Lambda^p}$ splits as $R^{\Lambda^p} = R_0^{\Lambda^p} + \frac{1+\delta}{2} G^{\Lambda^p}$. Using \eqref{equation:bk-pas-sharp}, the fact that $R_0(a,b)$ (being skew-symmetric) is diagonalisable over $\mathbb{C}$, inducing a diagonal basis for $R_0^{\Lambda^p}(a,b)$, together with \eqref{equation:algebra}, we easily see that 
\begin{equation}
\label{equation:rolambdap}
|\langle R_0^{\Lambda^p}(a,b)\omega,\tau\rangle|\leq \tfrac{2p}{3}(1-\delta) |a||b||\omega||\tau|,
\end{equation}
for every tangent vector $a$, $b$ and $p$-forms $\omega$, $\tau$.

Let $\Sym^2 TM \to M$ be the bundle of symmetric $2$-tensors over $M$. In the following, using the metric we will identify $\Sym^2 TM \to M$ with the bundle of symmetric endomorphisms of $M$, whose scalar product is given by
\[
\langle A,B\rangle_x := \Tr(A(x)B(x)).
\]
The action of the curvature map is extended to $\Sym^2 TM$ by the commutator action, namely for all $x \in M, C \in \Sym^2 T_xM$,
\[
R^{\Sym^2}(a,b)C = [R(a,b),C].
\]
Similarly, there is a splitting $R^{\Sym^2} = R^{\Sym^2}_0 + \frac{1+\delta}{2}G^{\Sym^2}$ and using \eqref{equation:bk-pas-sharp}, we obtain the estimate:
\begin{equation}
\label{equation:rosym2}
|\langle R_0^{\Sym^2}(a,b)C,D\rangle|\leq \tfrac{4}{3}(1-\delta) |a||b||C||D|,
\end{equation}
for all $x \in M, a,b \in T_xM, C,D \in \Sym^2 T_xM$.

\subsection{Fourier analysis in the fibers}

\label{ssection:fourier}

Further details on this paragraph can be found in \cite{Paternain-99}, \cite[Section 2]{Paternain-Salo-Uhlmann-15} and \cite[Section 5]{Cekic-Lefeuvre-20}. We emphasise that the origins of Fourier decomposition on the unit tangent bundle and raising/lowering operators go back to \cite{Croke-Sharafutdinov-98, Guillemin-Kazhdan-80-2, Guillemin-Kazhdan-80} in the framework of spectral rigidity.

\subsubsection{Trivial line bundle}

Let $\pi : SM \rightarrow M$ be the projection on the base, where $SM$ is the unit tangent bundle of $(M,g)$. There is a canonical splitting of the tangent bundle of $SM$ as:
\[
T(SM) =  \V \oplus \HH \oplus \R X,
\] 
where $X$ is the geodesic vector field, $\V := \ker d \pi$ is the vertical space and $\HH$ is the horizontal space defined in the following way. Consider the \emph{connection map} $\mc{K} : T(SM) \rightarrow TM$ defined as follows: let $(x,v) \in SM, w \in T_{(x,v)}(SM)$ and a curve $(-\eps,\eps) \ni t \mapsto z(t) \in SM$ such that $z(0)=(x,v), \dot{z}(0)=w$; write $z(t)=(x(t),v(t))$; then $\mc{K}_{(x,v)}(w) := \frac{Dv(t)}{dt}\big|_{t=0}$, where $\frac{D}{dt}$ denotes covariant derivative along the curve $t \mapsto x(t)$ (see for instance \cite[Proposition 2.2]{Do-Carmo-book} for the definition of $\frac{D}{dt}$). Then $\mathbb{H} := \ker \mc{K}$ and if we define the \emph{normal bundle} $\mc{N} \rightarrow SM$ whose fiber at $(x,v) \in SM$ is given by $\mc{N}(x,v) := \left\{v\right\}^\bot \subset T_xM$, then $d \pi : \HH \rightarrow \left\{v\right\}^\perp, \mc{K} : \V  \rightarrow \left\{v\right\}^\perp$ are both identified with isomorphisms $d \pi : \HH \to \mc{N}, \mc{K} : \V \to \mc{N}$. In particular, we will think of the normal bundle $\mc{N}$ as the tangent bundle to the spheres. We denote by $g_{\mathrm{Sas}}$ the Sasaki metric on $SM$, which is the canonical metric on the unit tangent bundle, defined by:
\[
g_{\mathrm{Sas}}(w,w') := g(d \pi(w), d\pi(w')) + g(\mc{K}(w),\mc{K}(w')).
\]

For $x \in M$, the unit sphere
\[
S_xM = \left\{ v \in T_xM ~|~ |v|^2_x = 1\right\} \subset SM
\]
(endowed with the Sasaki metric) is then isometric to the canonical sphere $(\Ss^{n-1},g_{\mathrm{can}})$. We denote by $\Delta_{\V}$ the vertical Laplacian defined for $f \in C^\infty(SM)$ as $\Delta_{\V}f(x,v): = \Delta_{g_{\mathrm{can}}}(f|_{S_xM})(v)$, where $\Delta_{g_{\mathrm{can}}}$ is the (positive) spherical Laplacian. For $k \geq 0$, we introduce
\[
\Omega_k(x) = \ker(\Delta_{\V}(x) - k(n+k-2)),
\]
the spherical harmonics of degree $k$. Observe that $\Omega_k \rightarrow M$ is a well-defined vector bundle over $M$. Given $f \in C^\infty(SM)$, it can be decomposed as $f = \sum_{k \geq 0} \widehat{f}_k$ where $\widehat{f}_k \in C^\infty(M,\Omega_k)$ is the projection of $f$ onto spherical harmonics of degree $k$. We call \emph{Fourier degree} of $f$, denoted by $\mathrm{deg}(f)$, the maximal integer $k_0 \in \Z_{\geq 0}$ such that $\widehat{f}_{k_0} \neq 0$; it takes values in $\left\{0,1,\ldots,+\infty\right\}$. We will also say that $f$ has \emph{finite Fourier content} if its degree is finite, and that $f$ is \emph{odd} (resp. \emph{even}) if it only contains odd (resp. even) spherical harmonics. 

It can be proved that the operator $X$ has the following mapping properties (see \cite[Section 3]{Paternain-Salo-Uhlmann-15}):
\begin{equation}\label{eq:X+-}
X : C^\infty(M,\Omega_k) \rightarrow C^\infty(M,\Omega_{k+1}) \oplus C^\infty(M,\Omega_{k-1}).
\end{equation}
This is understood as follows: a section $\widehat{f}_k \in C^\infty(M,\Omega_k)$ defines in particular a smooth function in $C^\infty(SM)$; we can differentiate in the $X$-direction and this only contains spherical harmonics of degree $k-1$ and $k+1$. Taking the projection on higher degree (resp. lower degree), we obtain an operator $X_+ : C^\infty(M,\Omega_k) \rightarrow C^\infty(M,\Omega_{k+1})$ of gradient type, i.e. with injective principal symbol (resp. $X_- : C^\infty(M,\Omega_k) \rightarrow C^\infty(M,\Omega_{k-1})$ of divergence type) such that $X=X_+ + X_-$ and $X_+^*=-X_-$ (the latter being a mere consequence of the fact that $X^*=-X$, since $X$ preserves the Liouville measure on $SM$). Here, $P^*$ denotes the formal $L^2$-adjoint of an operator $P$. As $X_+$ acting on spherical harmonics of degree $k$ has injective principal symbol, its kernel is finite dimensional by elliptic theory if $M$ is compact. We call \emph{conformal Killing tensors} of degree $k \in \Z_{\geq 0}$ the elements in its kernel.

\subsubsection{Twist by a vector bundle}

\label{sssection:twisted-discussion}

Let $\mc{E} \rightarrow M$ be a \emph{real}\footnote{It can also be taken to be complex but we will always consider real bundles throughout this article.} vector bundle over $M$ equipped with an orthogonal connection $\nabla^{\E}$. Consider the pullback bundle $\pi^*\mc{E} \rightarrow SM$ equipped with the pullback connection $\pi^*\nabla$ and introduce the first-order differential operator
\[
\X := (\pi^* \nabla^{\E})_X : C^\infty(SM,\pi^*\E)  \rightarrow C^\infty(SM,\pi^*\E).
\]

The connection $\pi^*\nabla^{\E}$ also gives rise to differential operators:
\[
\nabla^{\E}_{\HH,\V} : C^\infty(SM,\pi^*\E) \rightarrow C^\infty(SM,\pi^*\E \otimes \mc{N}),
\]
defined in the following way: given $f \in C^\infty(SM,\pi^*\E)$, the covariant derivative $\pi^*\nabla^{\E} f \in C^\infty(SM, T^*(SM) \otimes \pi^*\E)$ can be identified with an element of $C^\infty(SM,T(SM) \otimes \pi^*\E)$ by using the musical isomorphism $T^*(SM) \rightarrow T(SM)$ induced by the Sasaki metric. Using the orthogonal projections of $T(SM)$ onto $\HH$ and $\V$ one can then define the operators: 
\[
\nabla^{\E}_{\HH} f := d\pi( (\pi^*\nabla^{\E} f)_\HH), ~~~ \nabla^{\E}_{\V} f := \mc{K} ((\pi^*\nabla^{\E} f)_\V),
\]
which take values in the bundle $\pi^*\E \otimes \mc{N} \rightarrow SM$. In local coordinates, these operators have explicit expressions in terms of the connection $1$-form and we refer to \cite[Lemma 3.2]{Guillarmou-Paternain-Salo-Uhlmann-16} for further details.

If $(e_1, \ldots, e_r)$ is a smooth local orthonormal basis of $\E$ in a neighborhood of a point $x_0 \in M$, then smooth sections $f \in C^\infty(SM,\pi^*\E)$ can be written near $x_0$ as:
\[
f(x,v) = \sum_{j=1}^r f^{(j)}(x,v) e_j(x) \in \mc{E}_x,
\]
where $f^{(j)} \in C^\infty(SM)$ is only locally defined. As in \S\ref{ssection:fourier}, each $f^{(j)}$ can be in turn decomposed into spherical harmonics. In other words, we can write $f = \sum_{k \geq 0} \widehat{f}_k$, where $\widehat{f}_k \in C^\infty(M, \Omega_k \otimes \mc{E})$ and pointwise in $x \in M$:
\[
\Omega_k \otimes \mc{E} (x) := \ker(\Delta^{\mc{E}}_{\V}(x) - k(n+k-2)),
\]
is the kernel of the vertical Laplacian $\Delta^{\mc{E}}_{\V}$ (this Laplacian is independent of the connection $\nabla^{\mc{E}}$, it only depends on $\mc{E}$ and on $g$). Elements in this kernel are called the \emph{twisted spherical harmonics of degree $k$} and they form a well-defined vector bundle $\Omega_k \otimes \mc{E} \rightarrow M$. As in \S\ref{ssection:fourier}, we can define the degree of $f \in C^\infty(SM,\mc{E})$ and we say that $f$ has \emph{finite Fourier content} if its expansion in spherical harmonics only contains a finite number of terms.

We call \emph{twisted cohomological equation} an equation of the form $\X f = h$, for some given $h \in C^\infty(SM,\pi^*\E)$. We will be interested more specifically in the case where $h=0$. Similarly to \eqref{eq:X+-}, the operator $\X$ maps
\begin{equation}
\label{eq:XX}
\X : C^\infty(M, \Omega_k \otimes \mc{E}) \rightarrow C^\infty(M, \Omega_{k-1} \otimes \mc{E}) \oplus C^\infty(M, \Omega_{k+1} \otimes \mc{E})
\end{equation}
and can be decomposed as $\X = \X_+ + \X_-$, where, if $u \in C^\infty(M,\Omega_k \otimes \mc{E})$, $\X_+u \in C^\infty(M,\Omega_{k+1} \otimes \mc{E})$ denotes the orthogonal projection on the twisted spherical harmonics of degree $k+1$. The operator $\X_+$ is elliptic and thus has finite-dimensional kernel (when $M$ is compact) which consists of \emph{twisted conformal Killing tensors} (CKTs) whereas $\X_-$ is of divergence type. Moreover, $\X_+^* = -\X_-$, where the adjoint is computed with respect to the canonical $L^2$ scalar product on $SM$ induced by the Sasaki metric. We also refer to the original articles of Guillemin-Kazhdan \cite{Guillemin-Kazhdan-80, Guillemin-Kazhdan-80-2} for a description of these facts and to \cite{Guillarmou-Paternain-Salo-Uhlmann-16} for a more modern exposition. It was shown in \cite[Theorem 4.1]{Guillarmou-Paternain-Salo-Uhlmann-16} that flow-invariant sections, i.e. smooth sections in $\ker \X$ have \emph{finite Fourier content}, under the assumption that $(M, g)$ has negative sectional curvature. We mention that finiteness of Fourier content of flow-invariant sections is an open question under the assumption that the geodesic flow of $(M, g)$ is \emph{Anosov}.

Let us also mention that if $\mathfrak{o}(\E)$ is any vector bundle obtained by a functorial operation $\mathfrak{o}$ from $\E$ (e.g. dual, exterior and symmetric powers, tensor products), there is an induced orthogonal connection $\nabla^{\mathfrak{o}(\E)}$ on $\mathfrak{o}(\E)$ and thus an induced operator $\X$ acting on $C^\infty(SM,\pi^*\mathfrak{o}(\E))$. In order not to burden the notation, we will keep the notation $\X$ even if it might denote an operator acting on distinct vector bundles. In particular, this will be applied with $\mathfrak{o}(\E)= \Lambda^p \E$ or $\mathfrak{o}(\E)= \Sym^2 \E$, with $\E=TM$.

\subsubsection{Twisted Pestov identity}

The Pestov identity is a classical identity in Riemannian geometry, see \cite{Guillemin-Kazhdan-80,Croke-Sharafutdinov-98,Paternain-Salo-Uhlmann-15} and \cite{Guillarmou-Paternain-Salo-Uhlmann-16} for the twisted version. If $(\E,\nabla^{\E})$ is a vector bundle equipped with a smoothly varying inner product in its fibers, and with an orthogonal connection (i.e. compatible with the inner product structure on $\E$), we write $\End_{\mathrm{sk}}(\E)$ for skew-symmetric endomorphisms of $\E$ and 
\[
F_{\nabla} = F_{\nabla^{\E}} = (\nabla^{\E})^{ 2} \in C^\infty(M, \Lambda^2T^*M \otimes \End_{\mathrm{sk}}(\E)),
\]
for the curvature. Here, similarly to \S \ref{ssection:curvature-tensor-bounds}, $(\nabla^{\E})^{ 2}$ is defined as the square of the covariant derivative, where the second $\nabla^{\E}$ is defined using the Leibniz rule. Following \cite[Section 3]{Guillarmou-Paternain-Salo-Uhlmann-16}, let $\mc{F}^{\E} \in C^\infty(SM,\mc{N} \otimes \End_{\mathrm{sk}(\E))}$ be defined by the identity:
\begin{equation}
\label{equation:twisted-curvature}
\langle \mc{F}^{\E}(x,v)e, w \otimes e' \rangle := \langle {(F_{\nabla})}_x(v,w)e, e' \rangle,
\end{equation}

where $(x,v) \in SM, e, e' \in \E_x, w \in \mc{N}(x,v)$, and the metric on the right-hand side is the tensor product metric on $\mc{N}(x,v) \otimes \E_x$. Similarly, we will view the Riemannian curvature tensor as an operator on $\mc{N} \otimes \E$, defined by the relation:
\[R(x, v) (w \otimes e) = (R_x(w, v)v) \otimes e, \quad w \in \mc{N}(x, v), e \in \E_x.\]
\begin{lemma}
We have, for any orthonormal basis $\e_1, \dotso, \e_n \in T_xM$, and $v \in S_xM$:
\begin{equation}\label{eq:F^E}
	\mc{F}^{\E}(x, v) = \sum_{i = 1}^n \e_i \otimes F_{\nabla}(v, \e_i).
\end{equation}
\end{lemma}

\begin{proof} We simply write:
\[
\mc{F}^{\E}(x, v) = \sum_{i = 1}^n \e_i \otimes \mc{F}_i(x, v),
\]
for some endomorphisms $\mc{F}_i(x, v): \E_x \to \E_x$. Since $\mc{F}^{\E}$ is a section of $\mc{N} \otimes \End(\E)$, for all $B \in \End(\E_x)$ we have:
\[0 = \langle{\mc{F}^{\E}(x, v), v \otimes B}\rangle = \Big\langle{\sum_i \langle{v, \e_i}\rangle.\mc{F}_i, B}\Big\rangle,\]
that is we have $\sum_i \langle{v, \e_i}\rangle.\mc{F}_i(x, v) = 0$. Thus we compute, by plugging in $w = \e_j - v. \langle{v, \e_j}\rangle \in \mc{N}(x, v)$ in \eqref{equation:twisted-curvature}, for any $e, e' \in \E_x$:
\begin{align*}
	\langle F_{\nabla}(v, \e_j) e, e'\rangle &= \langle \mc{F}^{\E} (x, v) e, (\e_j - v. \langle{v, \e_j}\rangle) \otimes e'\rangle = \sum_i \Big\langle{\e_i \otimes \mc{F}_i e, (\e_j - v. \langle{\e_j, v}\rangle) \otimes e'}\Big\rangle\\
	&= \langle\mc{F}_j(x, v) e, e'\rangle - \langle{\e_j, v}\rangle. \Big\langle{\underbrace{\sum_i \langle{\e_i, v}\rangle.\mc{F}_i}_{= 0}e, e'}\Big\rangle,
\end{align*}
using that $F_{\nabla}(v, v) = 0$. Since $e, e' \in \E_x$ are arbitrary, we have $\mc{F}_j(x, v) = F_{\nabla}(v, \e_j)$.
\end{proof}

All the norms below are the $L^2$-norms. In order to avoid repetitions, we suppress the subscript $L^2$. We call the following identity, the \emph{twisted Pestov identity}. It is slightly different from what \cite{Guillarmou-Paternain-Salo-Uhlmann-16} call a twisted identity but the following lemma can be easily recovered from \cite[Proposition 3.5]{Guillarmou-Paternain-Salo-Uhlmann-16}.

\begin{lemma}[Localized Pestov identity]
\label{lemma:pestov}
Let $(M,g)$ be an $n$-dimensional Riemannian manifold. The following identity holds: for all $k \in \Z_{\geq 0}$, $u \in C^\infty(M,\Omega_k \otimes \E)$,
\begin{equation}
\label{equation:local-pestov}
\begin{split}
\dfrac{(n+k-2)(n+2k-4)}{n+k-3} \|\X_-u\|^2 -  & \dfrac{k(n+2k)}{k+1} \|\X_+u\|^2 + \|Z(u)\|^2 \\
& \hspace{3cm} = \langle R\nabla_{\V}^{\E}u, \nabla_{\V}^{\E}u \rangle + \langle \mc{F}^{\E}u, \nabla_{\V}^{\E}u \rangle,
\end{split}
\end{equation}
where $Z$ is a first order differential operator which we do not make explicit.
\end{lemma}

\begin{proof}
By \cite[Proposition 3.5]{Guillarmou-Paternain-Salo-Uhlmann-16}, the following equality holds for $u \in C^\infty(M,\Omega_k \otimes \E)$:
\begin{equation}
\label{equation:pestov-gpsu}
(n+2k -3) \|\X_-u\|^2 + \|\nabla_{\HH}^{\E}u\|^2 - \langle R\nabla_{\V}^{\E}u, \nabla_{\V}^{\E}u \rangle - \langle \mc{F}^{\E}u, \nabla_{\V}^{\E}u \rangle = (n+2k - 1) \|\X_+u\|^2.
\end{equation}
Moreover, by \cite[Lemma 3.7]{Guillarmou-Paternain-Salo-Uhlmann-16}, we have:
\[
\nabla_{\HH}^{\E} u = \dfrac{1}{k+1} \nabla^{\E}_{\V} \X_+ u - \dfrac{1}{n+k-3} \nabla^{\E}_{\V} \X_- u + Z(u),
\]
where $Z(u)$ is some term with vanishing vertical divergence (i.e. $\mathrm{div}^{\E}_{\V} Z(u) = 0$, where $\mathrm{div}^{\E}_{\V}$ is the formal adjoint to the vertical gradient $\nabla_{\V}^{\mc{E}}$). As a consequence, taking the $L^2$-norms, we get:
\begin{equation}
\label{equation:hterm}
\begin{split}
\|\nabla_{\HH}^{\E} u\|^2 & = \dfrac{1}{(k+1)^2} \|\nabla^{\E}_{\V} \X_+ u\|^2 + \dfrac{1}{(n+k-3)^2} \|\nabla^{\E}_{\V} \X_- u\|^2 + \|Z(u)\|^2\\
& = \dfrac{n+k-1}{k+1} \| \X_+ u\|^2 + \dfrac{k-1}{n+k-3} \|\X_-u\|^2 + \|Z(u)\|^2.
\end{split}
\end{equation}
(Here, we simply use that $\mathrm{div}^{\E}_{\V} \nabla^{\E}_{\V} = \Delta_{\V}^{\E}$.) Plugging \eqref{equation:hterm} into \eqref{equation:pestov-gpsu}, and after some algebraic simplifications, we obtain the claimed result.
\end{proof}

\subsection{Link with symmetric tensors}

We refer to \cite{Heil-Moroianu-Semmelmann-16} as well as \cite[Section 3]{Guillarmou-Paternain-Salo-Uhlmann-16} and \cite{Cekic-Lefeuvre-20} for further details on this paragraph. In the following, we will keep identifying $TM$ and $T^*M$ by the metric. Let $\Sym^k TM \to M$ be the vector bundle of symmetric $k$-tensors over $M$. If $u = \sum_{i_1,\ldots,i_k=1}^n u_{i_1 \ldots i_k} \e_{i_1} \otimes \ldots \otimes \e_{i_k} \in TM^{\otimes k}$, where $(\e_1,\ldots,\e_n)$ is a local orthonormal frame, then the orthogonal projection of $u$ onto symmetric tensors $\Sym : TM^{\otimes k} \to \Sym^k TM$ is given by (here $S_k$ denotes the permutation group):
\[
\Sym(u) = \dfrac{1}{k!} \sum_{\sigma \in S_k} \sum_{i_1,\ldots,i_k=1}^n u_{i_1 \ldots i_k} \e_{i_{\sigma(1)}} \otimes \ldots \otimes \e_{i_{\sigma(k)}}.
\]
Let $\mc{I} : \Sym^{k+2} TM  \to \Sym^k TM$ be the trace operator defined pointwise on $M$ as:
\[
\mc{I} u := \sum_{i=1}^n u(\e_i,\e_i, \cdot, \ldots,\cdot).
\]
We define $\Sym^k_0 TM := \Sym^k TM \cap \ker \mc{I}$ to be the space of trace-free tensors. The adjoint of $\mc{I}$ with respect to the natural metric\footnote{The scalar product on $TM^{\otimes k}$ is given by:
\[
g_{TM^{\otimes k}}(v_1 \otimes \ldots \otimes v_k, w_1 \otimes \ldots \otimes w_k) := \prod_{j=1}^k g(v_j,w_j),
\] 
where $v_i,w_i \in TM$ and this induces a scalar product on $\Sym^kTM$ by restriction of the metric.} on $\Sym^k TM$, which we denote by $\mc{J}$, is the symmetric multiplication by the metric $g$, namely:
\[
\mc{I}^* u = \mc{J} u = \Sym(g \otimes u).
\]
The space $\Sym^k TM$ breaks up as the orthogonal sum:
\begin{equation}
\label{equation:decomposition}
\Sym^k TM = \oplus_{i=0}^{\floor{k/2}} \mc{J}^{i} \Sym^{k-2i}_0 TM.
\end{equation}
Let
\[
\mc{P} : \Sym^k TM  \to \Sym^k_0 TM
\]
be the orthogonal projection onto trace-free symmetric tensors, that is, onto the highest degree summand of \eqref{equation:decomposition}.

We can consider symmetric tensors in $\Sym^k TM$ as homogeneous polynomials of degree $k$ on $TM$, or, by restricting to the unit sphere, as spherical harmonics of degree $\leq k$. In fact, for any $x \in M$, $k \in \Z_{\geq 0}$, we introduce the \emph{pullback operator} defined pointwise at $x$ by:
\begin{equation}
\label{equation:pullback}
(\pi^*_k u) (x,v) := u_x(v^{\otimes k}), \hspace{1cm} \pi^*_k : \Sym^k TM(x) \otimes \E(x) \xrightarrow{\sim} \bigoplus_{i=0}^{\floor{k/2}} \Omega_{k-2i}(x) \otimes \E(x),
\end{equation}
and this operator is a graded isomorphism
\[
\pi^*_k : \Sym^k_0 TM (x) \otimes \E (x) \xrightarrow{\sim} \Omega_k(x) \otimes \E(x).
\]
Given $f \in C^\infty(SM,\pi^*\E)$, we have $\deg(f) < \infty$ if and only if there exists $k_1, k_2 \in \Z_{\geq 0}$, $f_{\mathrm{even}} \in C^\infty(M,\Sym^{2 k_1}TM  \otimes \E), f_{\mathrm{odd}} \in C^\infty(M,\Sym^{2 k_2+1} TM \otimes \E)$ such that
\[
f = \pi^*_{2k_1} f_{\mathrm{even}} + \pi^*_{2k_2 + 1} f_{\mathrm{odd}}.
\]
In other words, sections on $SM$ with finite Fourier content can always be seen as sections defined on the base. We now relate the operators $\X$ and $\X_+$ with the usual \emph{symmetrized covariant derivative}
$\dd$ and $\dd_{\E}$. We introduce:
\[
\dd : C^\infty(M,\Sym^k TM) \to C^\infty(M,\Sym^{k+1} TM), ~~~~ \dd := \Sym \circ \nabla,
\]
where $\nabla$ is the Levi-Civita connection induced by $g$. Similarly, in the twisted case, we can consider:
\[
\dd_{\E} : C^\infty(M,\Sym^k TM \otimes \E) \to C^\infty(M,\Sym^{k+1} TM \otimes \E), ~~~~ \dd_{\E} := \Sym \circ \nabla^{\E}.
\]
For the sake of simplicity, we will drop the subscript $\E$ and write $\dd$, even in the twisted case. We have the following relations, see \cite[Section 3.6]{Guillarmou-Paternain-Salo-Uhlmann-16}:

\begin{lemma}
\label{lemma:link}
The following relations hold:
\begin{enumerate}
\item For $f \in C^\infty(M,\Sym^k TM \otimes \E)$, $\pi_{k+1}^* \dd f = \X \pi_k^* f$;
\item For $f \in C^\infty(M,\Sym^k_0 TM \otimes \E)$, $\pi_{k+1}^* \mc{P} \dd f  = \X_+ \pi_k^* f$.
\end{enumerate}
\end{lemma}

Elements in $\ker \dd$ are called \emph{twisted Killing tensors} while elements in $\ker \mc{P} \dd$ are called \emph{twisted conformal Killing tensors}. We remark that while in the non-twisted case $\mc{P} \dd$ is equivariant under conformal rescalings of the metric (see \cite[Section 3.6]{Guillarmou-Paternain-Salo-Uhlmann-16}), the connection $\nabla^{\E}$ on the auxiliary bundles considered in this paper depends on the metric on $M$, and it is not entirely clear if the conformal equivariance is preserved. Despite this, and by analogy with the classical setting, we will still call elements of $\ker \mc{P} \dd$ twisted conformal Killing tensors. In the non-twisted case, when the metric has negative sectional curvatures, it is known that there are no non-trivial conformal Killing tensors for 
$k \geq 1$, see \cite{Dairbekov-Sharafutdinov-10, Heil-Moroianu-Semmelmann-16}. 

Investigating the (non-)existence of twisted conformal Killing tensors satisfying certain algebraic properties will play an important role in this article and this is due to the following observation: if $f \in C^\infty(SM,\pi^*\E) \cap \ker \X$, then $f$ has finite degree by \cite[Theorem 4.1]{Guillarmou-Paternain-Salo-Uhlmann-16}. Hence, we can decompose $f$ into Fourier modes $f=\widehat{f}_0 + \ldots + \widehat{f}_k$ with $k \geq 0$, $\widehat{f}_k \neq 0$, and we have $\X_+ \widehat{f}_k = 0$ by the mapping properties \eqref{eq:XX} of $\X$. By the previous Lemma \ref{lemma:link}, this implies the existence of a non-trivial trace-free twisted conformal Killing tensor of degree $k \geq 0$.

\section{Dynamics of the frame flow and topology of the frame bundle}

\label{section:dynamics}

The aim of this section is twofold: firstly, we introduce the \emph{transitivity group} in the context of principal bundle extensions of Anosov flows and state Proposition \ref{proposition:invariant}, which gives a bijection between the space of flow-invariant sections of (associated) vector bundles and fixed vectors of the (associated) representation action of the transitivity group. (This also appears in the context of non-Abelian Liv\v{s}ic theory, see \cite[Theorem 1.3 and Lemma 3.7]{Cekic-Lefeuvre-21-1}.) Secondly, in the case of the frame bundle of a negatively curved manifold we rule out most of such groups using the topology of structure group reductions of the unit sphere; for the remaining ones we construct corresponding flow-invariant structures over $SM$ and study their basic properties.

\subsection{Extensions of Anosov flows to principal bundles}

\label{ssection:dynamics1}

Frame flows on negatively curved manifolds are typical examples of partially hyperbolic systems arising as the extension on a principal bundle of an Anosov flow. This was originally studied by Brin \cite{Brin-75-1,Brin-75-2} (see also the survey \cite{Brin-82}) who treated the general case of a principal bundle over a manifold with transitive Anosov flow. In this paragraph, we consider the following setting: we let $\mc{M}$ be a smooth closed manifold equipped with a volume-preserving Anosov flow $(\varphi_t)_{t \in \R}$ and we let $X \in C^\infty(\mc{M}, T\mc{M})$ be its generator, $\pi : P \to \mc{M}$ is a principal $G$-bundle over $\mc{M}$ and $(\Phi_t)_{t \in \R}$ is an extension of $(\varphi_t)_{t \in \R}$ in the sense that it satisfies the relations: for all $t \in \R, g \in G$,
\begin{equation}
\label{equation:relations}
\pi \circ \Phi_t = \varphi_t \circ \pi, ~~~~ R_g \circ \Phi_t = \Phi_t \circ R_g,
\end{equation}
where $R_g : P \to P$ denotes the right-action in the fibers. 

In order to describe the flow $(\Phi_t)_{t \in \R}$, we adopt a slightly different point of view than the original approach of Brin \cite{Brin-75-1,Brin-75-2} and use the point of view of \cite{Cekic-Lefeuvre-21-1}. Proofs of the following facts can be found in \cite{Lefeuvre-21}. We fix an arbitrary periodic orbit $\gamma_\star \subset \mc{M}$ for the flow $(\varphi_t)_{t \in \R}$, of period $T_\star$, and we fix a point $x_\star \in \gamma_\star$. We let $\mc{H}$ be the set of \emph{homoclinic orbits} to $\gamma_\star$, namely the set of all orbits for the flow $(\varphi_t)_{t \in \R}$ converging in the past and in the future to the periodic orbit $\gamma_\star$. We then introduce \emph{Parry's free monoid} $\mathbf{G}$ as the following formal free monoid:
\[
\mathbf{G} := \left\{ \gamma_1 \cdot \ldots \cdot \gamma_k \mid  k \in \N, \forall i \in \left\{1,\ldots,k\right\}, \gamma_i \in \mc{H} \right\}.
\]

We denote by $W^{s,u}_{\mc{M}}$ the strong stable/unstable foliation of the flow on $\M$. Given $x \in \M$ and $x' \in W_{\M}^{s}(x)$, one can define for $x$ close to $x'$ a holonomy map $\mathrm{Hol}^s_{x \to x'} : P_x \to P_{x'}$ in the following way:
\begin{equation}
\label{equation:limit}
\mathrm{Hol}^s_{x \to x'} w := \lim_{t \to +\infty} \Phi_{-t} \circ C_{\varphi_t (x) \to \varphi_t(x')} \circ \Phi_t(w),
\end{equation}
where $C_{x_1 \to x_2} : P_{x_1} \to P_{x_2}$ is defined, for pairs of points that are close enough, as the parallel transport (with respect to an arbitrary connection on $P$) along the unique short geodesic (with respect to an arbitrary metric on $\M$) joining $x_1$ to $x_2$. Convergence of \eqref{equation:limit} is ensured by the Ambrose-Singer formula together with the fact that the distance between $\varphi_t(x)$ and $\varphi_t(x')$ converges exponentially fast to $0$, see \cite[Section 3.2.2]{Cekic-Lefeuvre-21-1} for instance.  Alternatively, one can define the holonomy map $w' := \mathrm{Hol}^s_{x \to x'}w$ as the unique point in the intersection
\[
W^s_{P}(w) \cap P_{x'} = \left\{w'\right\},
\]
where $W^{s,u}_{P}$ denotes the strong stable/unstable foliation in the principal bundle. Similarly, we define $\mathrm{Hol}^u_{x \to x'}$ for $x'\in W^u_{\M}(x)$ by taking the limit as $t \to -\infty$. Eventually, for $x'$ on the same flowline as $x$, there is also a natural holonomy map $\mathrm{Hol}^c_{x \to x'}$ given by the flow $(\Phi_t)_{t \in \R}$ itself. Let $P_\star := P_{x_\star}$ be the fiber over $x_\star$. After an arbitrary choice of point $w_\star \in P_\star$, the fiber $P_\star$ gets identified with the group by the map $G \to P_\star, g \mapsto R_g w_\star$.

This formalism allows to define a representation of the monoid $\mathbf{G}$. We call \emph{Parry's representation} the representation $\rho : \mathbf{G} \to G$ of the free monoid obtained by the following process. For $w \in G \simeq P_\star$, we set $\rho(\gamma_\star)w = \Phi_{T_\star}(w)$, and for $\gamma \in \mc{H}$ with $\gamma \neq \gamma_\star$, we pick two arbitrary points $x_1 \in W^{u}_{\M}(x_\star) \cap \gamma, x_2 \in W^{s}_{\M}(x_\star) \cap \gamma$ (where $x_{1,2}$ are chosen close to $x_\star$) and define 
\begin{equation}
\label{equation:parry}
\rho(\gamma)w := \mathrm{Hol}^s_{x_2 \to x_\star} \circ \mathrm{Hol}^c_{x_1 \to x_2} \circ \mathrm{Hol}^u_{x_\star \to x_1}w \in G.
\end{equation}
Note that $\rho(\gamma)$ commutes with the right action by \eqref{equation:relations} so it is a left action and can therefore be identified with an element of the group $G$ itself, namely $\rho(\gamma) \in G$. (Note also that $\rho(\gamma)$ is an isometry of $G$ equipped with a bi-invariant Riemannian metric.)

In \cite{Cekic-Lefeuvre-21-1} Parry's representation was defined by choosing a $k_n \to \infty$ such that we have $\Phi_{T_\star}(x_\star)^{k_n} \to \id_{x_\star}$, and specialising to $t = k_n T_\star$ in \eqref{equation:limit}; in particular, the two definitions agree. The following notion plays a central role and was identified by Brin \cite{Brin-75-1,Brin-75-2}:

\begin{definition}
The \emph{transitivity group} of the flow $(\Phi_t)_{t \in \R}$ is $H := \overline{\rho(\mathbf{G})}$.
\end{definition}

Observe that as $H$ is a closed subgroup of a compact Lie group, it is thus a Lie group \cite[Theorem 2.3]{Helgason-01}. Moreover, $H$ is independent of the choice of points $x_1,x_2$ in \eqref{equation:parry}. However, the transitivity group $H$ does depend on the choice of point $w_\star$ in $P_\star$ and changing $w_\star$ by $w_\star'$, one obtains another group $H'$ which is conjugate to $H$. It is clear that changing $x_\star$ to $\varphi_t(x_\star)$ for some $t$ corresponds to conjugating $H$ by the flow $\Phi_t$. A straightforward exercise based on the Anosov Closing Lemma shows that changing the periodic orbit $\gamma_\star$ to $\gamma_\star'$ also changes $H$ by a conjugation which is defined similarly to \eqref{equation:parry} where instead of a homoclinic orbit, we take a flow orbit converging to $\gamma_\star'$ in the future and to $\gamma_\star$ in the past. In other words, the transitivity group is only well-defined up to conjugacy. We observe that Brin does not use exactly the same definition of transitivity group but the two notions coincide (see \cite{Lefeuvre-21} for more details, where this is not explicitly proved but relations to holonomy along concatenations of arbitrary \emph{us-paths} and flowlines are given, see e.g. proofs of \cite[Lemmas 3.2 and 3.3]{Lefeuvre-21}). In what follows, ergodicity is considered with respect to the product of the volume on $\mc{M}$ and the Haar measure in the fibres.

\begin{proposition}
\label{proposition:brin}
There exists an $H$-principal bundle $Q \to \M$ such that $w_\star \in Q$, $Q \subset P$ is a smooth flow-invariant subbundle, and the restriction of $(\Phi_t)_{t \in \R}$ to $Q$ is ergodic. In particular, if $H = G$, then the flow $(\Phi_t)_{t \in \R}$ is ergodic.
\end{proposition}

For a proof, we refer to \cite{Lefeuvre-21} (see also the papers of Brin \cite{Brin-75-1,Brin-75-2}). From a topological perspective, the reduction to a subgroup $H \leqslant G$ is a strong constraint called a \emph{reduction of the structure group} of the principal bundle (see \S \ref{ssection:topology}). We remark that in the particular case of compact quotients of the complex hyperbolic plane, the transitivity group is equal to $\mathrm{U}(\frac{n}{2} - 1)$, see \cite{Cekic-Lefeuvre-Moroianu-Semmelmann-24} and references therein.

\subsection{Transitivity group of the frame flow}

\label{ssection:dynamics2}

We now specify the previous discussion to the case where $\M = SM$, $X$ is the geodesic vector field and $(\varphi_t)_{t \in \R}$ is the geodesic flow, $(\Phi_t)_{t \in \R}$ is the frame flow on the principal $\mathrm{SO}(n-1)$-bundle $FM \to SM$. A point $w \in FM_{(x,v)}$ over $(x,v) \in SM$ is seen from now on as an isometry $w : \R^{n-1} \to v^\bot$. For $g \in \mathrm{SO}(n-1)$, the right-action $R_g$ in the fibers is given by composition, namely $R_g w = w \circ g$.

We fix an arbitrary periodic point $(x_\star,v_\star)$ and set $\mc{N}_{\star} := \mc{N}(x_\star,v_\star)$. By Proposition \ref{proposition:brin}, the closure $H$ of the representation of Parry's free monoid $\rho : \mathbf{G} \to \mathrm{SO}(\mc{N}_\star) \simeq \mathrm{SO}(n-1)$ gives a flow-invariant reduction of the structure group of the frame bundle. Note that the identification $\R^{n-1} \simeq \mc{N}_\star$ is made by choosing an arbitrary isometry $w_\star : \R^{n-1} \to \mc{N}_\star$, and this corresponds to choosing a point in the frame bundle over $(x_\star,v_\star)$. Changing $w_\star$ by another isometry $w_\star'$, we would obtain another conjugate group $H'$. We start by the following observation, mostly due to Brin \cite{Brin-75-1}:

\begin{lemma}
\label{lemma:brin}
Let $(M,g)$ be a negatively curved Riemannian manifold of dimension $\geq 3$. Let $H \leqslant \mathrm{SO}(n-1)$ be the transitivity group of the frame flow and $H_0 \leqslant H$ be its identity component. Then, there is a finite Riemannian cover $(\widehat{M},\widehat{g}) \to (M,g)$ such that the frame flow has transitivity group equal to $H_0$. 
\end{lemma}

The conclusion also holds by replacing the negatively curved assumption by Anosov.

\begin{proof}
Let $w_\star \in FM$ and $Q(w_\star)$ be its ergodic component given by Proposition \ref{proposition:brin}. Since the restriction of the flow to $Q(w_\star)$ is transitive, the bundle $Q(w_\star) \to SM$ is a connected $H$-principal bundle and thus $Q(w_\star)/H_0 \to SM$ is a finite cover of $SM$ with deck transformation group $G := H/H_0$. Now, by the long exact sequence in homotopy \cite[Theorem 4.41]{Hatcher-02}, we have:
\[
\ldots \longrightarrow \underbrace{\pi_1(\Ss^{n-1})}_{=0} \longrightarrow \pi_1(SM) \longrightarrow \pi_1(M) \longrightarrow \underbrace{\pi_0(\Ss^{n-1})}_{=0} \longrightarrow \ldots,
\]
that is $\pi_1(SM) \simeq \pi_1(M)$. Thus there is a subgroup $\Gamma \leqslant \pi_1(M)$ such that $\widehat{M} := \widetilde{M}/\Gamma$ is a Riemannian cover of $M$ (equipped with the pullback metric) with deck transformation group $G$ and a diffeomorphism $F : S\widehat{M} \to Q(w_\star)/H_0$; here $\widetilde{M}$ denotes the universal cover of $M$. Moreover, the geodesic flow on $S\widehat{M}$ and the frame flow on $Q(w_\star)/H_0$ are $G$-equivariant and both project to the geodesic flow on $SM$ so they are conjugate by $F$. Finally, it suffices to observe that the frame flow on $S\widehat{M}$ has transitivity group equal to $H_0$. 
\end{proof}

By construction, the group $H \leqslant \mathrm{SO}(\mc{N}_\star)$ acts on $\mc{N}_\star$ and thus on any vector space obtained functorially from $\mc{N}_\star$. In other words, if $\mathrm{Vec}$ denotes the category of finite-dimensional vector spaces equipped with an inner product and
\[
	\mathfrak{o} : \mathrm{Vec} \to \mathrm{Vec}
\]
is an operation of (finite-dimensional) vector spaces (e.g. exterior power, symmetric or tensor power, dual), we obtain an induced representation
\[
	\rho_{\mathfrak{o}} : H \to \mathrm{SO}(\mathfrak{o}(\mc{N}_{\star})),\quad \rho_{\mathfrak{o}}(h) = \mathfrak{o}(\rho(h)),\quad h \in H,
\]
where $\rho$ is identified with the inclusion $H \xhookrightarrow{} \mathrm{SO}(\mc{N}_\star)$. In fact, it is easy to see that $\rho_{\mathfrak{o}}(H)$ is the image of Parry's representation as in \eqref{equation:parry}, where $P$ is replaced by the vector bundle $\mathfrak{o}(\mc{N})$, and the flow $(\Phi_t)_{t \in \mathbb{R}}$ is replaced by parallel transport along orbits of $(\varphi_t)_{t \in \mathbb{R}}$. We define:
\[
\mathfrak{f}_{\mathfrak{o}} := \left\{\omega \in \mathfrak{o}(\mc{N}_{\star}) \mid \forall h \in H, \rho_{\mathfrak{o}}(h)\omega = \omega\right\},
\]
i.e. these are all the elements of $\mathfrak{o}(\mc{N}_\star)$ which are invariant by the induced action of $H$.

\begin{proposition}
\label{proposition:invariant}
Let $\mathfrak{o} : \mathrm{Vec} \to \mathrm{Vec}$ be an operation of vector spaces. Then, the evaluation map:
\[
\Phi : C^\infty(SM, \mathfrak{o}(\mc{N})) \cap \ker \X \to \mathfrak{f}_{\mathfrak{o}}, ~~~ \Phi(f) := f(x_\star,v_\star),
\]
is an isomorphism. 
\end{proposition}

In other words, if the induced representation on $\mathfrak{o}(\mc{N}_\star)$ fixes a vector, then there exists a unique flow-invariant section of $\mathfrak{o}(\mc{N}) \to SM$ whose value at $(x_\star,v_\star)$ is given by that vector. We only sketch the proof, which can be found in \cite[Lemma 3.7]{Cekic-Lefeuvre-21-1}. The presence of the operation $\mathfrak{o}$ is purely functorial, so it is harmless to assume that $\mathfrak{o}=\mathbbm{1}$ (identity) and to work with $\mc{N}_\star$.

\begin{proof}[Idea of proof]
The fact that the map $\Phi$ is well-defined is almost tautological. The injectivity of $\Phi$ is immediate. Indeed, if $\X f = 0$, we have $X|f|^2 = 2 \langle \X f, f \rangle = 0$ that is $|f|^2$ is constant by ergodicity of the geodesic flow. Hence if $\Phi(f) = f(x_\star,v_\star)=0$, we get $f=0$. Let us now show surjectivity. We consider a vector $e_\star \in \mc{N}_\star$ such that for any homoclinic orbit $\gamma \in \mathbf{G}$, we have $\rho(\gamma) e_\star = e_\star$. We denote by $(\Phi_t^{(2)})_{t \in \R}$ the flow on $2$-frames (i.e. the parallel transport along the geodesic flow of a vector that is orthogonal to the geodesic). Given $\gamma \in \mathbf{G}$, we can then define the invariant section $f$ on $\gamma$ in the following way: consider an arbitrary point $(x_0,v_0) \in \gamma \cap W^{u}_{SM}(x_\star,v_\star)$ and define for $(x,v) \in \gamma$, $f(x,v) := \Phi_t^{(2)} \circ \mathrm{Hol}^u_{(x_\star, v_\star) \to (x_0,v_0)} e_\star$, where $t \in \R$ is such that $\varphi_t(x_0,v_0) = (x,v)$. Using that $e_\star$ is preserved by the holonomy along $\gamma_\star$, one can check that this definition is independent of the choice of $(x_0,v_0)$. Moreover, $f$ is obviously flow-invariant on $\gamma$ by construction.

The union of all homoclinic orbits turns out to be dense in $SM$ by the Anosov Closing Lemma. Following the same proof as in \cite[Lemma 3.21]{Cekic-Lefeuvre-21-1} relying on the local product structure of the geodesic flow, one can then show that $f$ is Lipschitz-continuous on the set of all homoclinic orbits (the key point in this proof is that $e_\star$ satisfies $h \cdot e_\star = e_\star$ for all $h \in H$; this implies that $f(x,v)$ is also equal to $\Phi_{t'}^{(2)} \circ \mathrm{Hol}^s_{(x_\star, v_\star) \to (x_1,v_1)} e_\star$, where $(x_1,v_1) \in \gamma \cap W^s_{SM}(x_\star,v_\star)$ and $t' \in \R$ is such that $(x,v) = \varphi_{t'}(x_1,v_1)$). Since homoclinic orbits are dense in $SM$, this shows that $f$ extends as a Lipschitz-continuous section $f \in C^{\mathrm{Lip}}(SM, \mc{N})$ that is flow-invariant, i.e. such that $\X f = 0$. Since the \emph{threshold}\footnote{The threshold in this case is the exponential growth of the norm of the propagator $e^{t\X}$ acting on $L^\infty(SM,\mc{N})$. The connection being orthogonal, this operator is unitary, that is $\|e^{t\X}\|_{L^\infty \to L^\infty} \leq 1$.} of the operator $\X$ is $0$, we can apply \cite[Theorem 1.4]{Bonthonneau-Lefeuvre-21} to obtain that $f$ is actually smooth.
\end{proof}

To conclude, let us make the following important remark:

\begin{remark}
The \emph{type} of an $H$-invariant object is always preserved. For instance, if $n=8$, $H = \mathrm{G}_2$, $\mathfrak{o} = \Lambda^3$, then $H$ fixes an invariant vector $f_\star \in \Lambda^3 \mc{N}_\star$, the 3-form defining the $\mathrm{G}_2$-structure. Consider the section $f \in C^\infty(SM,\Lambda^3 \mc{N})$ such that $\X f = 0$, $f(x_\star,v_\star)=f_\star$, provided by Proposition \ref{proposition:invariant}. Then $f(x,v)$ is a $3$-form with $\mathrm{G}_2$-stabilizer on $\mc{N}(x,v)$ for all $(x,v) \in SM$. This can be easily seen by the following argument. First of all, in order to define $f_\star$, one has first chosen an implicit arbitrary isometry $w_\star : \R^{7} \to \mc{N}(x_\star,v_\star)$ and then $f_\star := w_\star(\xi_0)$, where $\xi_0$ is the $3$-form $\mathrm{G}_2$-structure on $\R^{7}$. The isometry $w_\star$ is also a point in the frame bundle over $SM$. By Proposition \ref{proposition:brin}, there exists a (unique) principal $\mathrm{G}_2$-bundle $Q \to SM$ which is flow-invariant and such that $w_\star \in Q$. Let $(x,v) \in SM$ and $w_{(x,v)} \in Q_{(x,v)}$ seen as an isometry $w_{(x,v)} : \R^{7} \to \mc{N}(x,v)$. Then we get an induced isometry $w_{(x,v)} : \Lambda^3 \R^{7} \to \Lambda^3 \mc{N}(x,v)$ and we claim that $f(x,v) = w_{(x,v)}(\xi_0)$. Indeed, this is clearly independent of the choice of point $w_{(x,v)} \in Q_{(x,v)}$ since any other point $w'_{(x,v)}$ is obtained as $w'_{(x,v)} = w_{(x,v)} \circ h$, where $h \in \mathrm{G}_2$ and $w'_{(x,v)}(\xi_0) = w_{(x,v)}(h\xi_0) = w_{(x,v)}(\xi_0)$. Moreover, this is flow-invariant since $Q$ is flow-invariant and $w_{(x,v)}(\xi_0)$ agrees with $f$ at $(x_\star,v_\star)$ so by Proposition \ref{proposition:invariant} they are equal.

\end{remark}

\subsection{Topological reduction of the structure group}

\label{ssection:topology}

We refer to \cite{Davis-Kirk-01} for the background in algebraic topology. On the sphere $\Ss^n$ (for $n \geq 2$), the topology of principal bundles is easier than on general CW-complexes since it does not require the use of classifying spaces. A principal $G$-bundle $P \to \Ss^n$ is determined by the homotopy class of its \emph{classifying map} $\tau : \Ss^{n-1} \to G$. We note here that $\tau$ is also known as the \emph{clutching map} in the literature and that it specifies how to glue two trivial $G$-principal bundles over upper and lower hemispheres of $\mathbb{S}^n$ along the equator $\Ss^{n-1}$. (The usual classifying map of $P$ from homotopy theory is a map from $\mathbb{S}^n$ to the classifying space of $G$; since the base is a sphere, such maps are in bijection with maps $\mathbb{S}^{n-1} \to G$, up to homotopy.) We will use the following standard terminology of algebraic topology:

\begin{definition}
\label{definition:structure}
Let $P \to \Ss^n$ be a principal $G$-bundle over $\Ss^n$, where $G$ is a compact Lie group. We say that $P$ admits a reduction\footnote{Note that in the literature, the letter for the group reduction is $G$ and the usual terminology is that of a \emph{$G$-structure}. In order to avoid a notational clash with the letter used for the transitivity group, we stick to the letter $H$.} of its structure group to $(H,\rho)$, where $H$ is a compact Lie group and $\rho : H \to G$ is a homomorphism, if the classifying map $\tau : \Ss^{n-1} \to G$ can be factored (up to homotopy) through the map $\rho$, that is, there exists a classifying map $\tau_0$ such that the following diagram commutes up to homotopy: 
\[
\xymatrix{
     & H \ar[d]^{\rho} \\
    \Ss^{n-1} \ar[ur]^{\tau_0} \ar[r]_\tau & G
  }
\]
\end{definition}

Another accepted (but not exactly equivalent) definition is to restrict only to subgroups $H \leqslant G$ and then $\rho = \iota$ where $\iota : H \to G$ is the embedding of $H$ into $G$.
If we had added the requirement that $\rho$ is injective in Definition \ref{definition:structure} these would be equivalent, but it is sometimes useful to allow for non-injective homomorphisms $\rho$. 
We will need the following:

\begin{lemma}
\label{lemma:lift}
Assume that $n\ge 3$ and let $\tau : \Ss^{n-1} \to G$ be a principal $G$-bundle over $\Ss^n$, identified with its classifying map. If $(H,\rho)$ is a reduction of the structure group, $H$ is a compact semisimple\footnote{Recall that a compact connected Lie group $H$ is said to be semisimple if it possesses no Abelian connected normal subgroups other than $\left\{1\right\}$.} Lie group, and $\pi : \widetilde{H} \to H$ is a cover of $H$, then $(\widetilde{H},\rho \circ \pi)$ is a reduction of the structure group.
\end{lemma}

\begin{proof}
Since $\pi_1(\Ss^{n-1}) = 0$, the existence of such a lift is immediate.
\end{proof}

Recall that a section $f \in C^\infty(SM,\pi^*\E)$ is said to be \emph{odd} or \emph{even} if its Fourier degrees are either all odd or even and invariant (or flow-invariant) if it satisfies $\X f = 0$. If $(M,g)$ has negative sectional curvature (or, more generally, transitive geodesic flow) any flow-invariant object has constant norm on $SM$ since $X|f|^2 = 2 \langle \X f, f \rangle = 0$ and the geodesic flow is ergodic. A typical flow-invariant odd object on the unit tangent bundle is the \emph{tautological section} $s \in C^\infty(SM,\pi^*TM)$ given by $s(x,v) := v$. It can be easily checked that it satisfies $\X s = 0$. In order not to burden the notation, we will simply denote it by $v$. 

Other natural flow-invariant sections appear in the setting of reduced holonomy. Let us give a motivating example in the case when $(M, g)$ has a compatible K\"ahler structure $J$. Then $f_1(x, v): = J_x v$ is a flow-invariant section of $\mc{N}$ of degree $1$, and $f_2(x, v): = \langle{ J_x v,\cdot}\rangle J_x v$ is a flow-invariant section of $\Sym^2\mc{N}$ of degree $2$, like in the second item of the theorem below (cf. also  \cite[Section 4.2]{Cekic-Lefeuvre-Moroianu-Semmelmann-22}). Similarly, if $(M,g)$ has a quaternion-Kähler structure, e.g. compact quotients of the quaternionic hyperbolic space, with Kraines form (see \cite[14.92]{Besse} or \cite{Kraines}) denoted by $\Omega\in C^\infty({M},\Lambda^4 TM)$, then $f_3(x,v):=\imath_v \Omega$, the contraction of $\Omega$ with $v$, is a flow-invariant section of $\Lambda^3\mc{N}$ of degree $1$.

The following statement shows that the non-ergodicity of the frame flow implies the existence of new invariant objects on the unit tangent bundle.
Recall that the Radon-Hurwitz numbers $\rho(n)$ are defined as follows: writing $n = (2a+1)2^b$, $b = c + 4d$ with $0 \leq c \leq 3$, we have $\rho(n) := 2^c + 8d$. The number $\rho(n)-1$ corresponds to the maximal number of linearly independent vector fields on the sphere $\Ss^{n-1}$, see \cite{Adams-62}. Note that $\rho(8)-1=7$, and for $\ell \in \mathbb{Z}_{\geq 0}$, we have $\rho(4\ell+2)-1 = 1, \rho(8\ell+4)-1=3$.

\begin{theorem}
\label{theorem:reduction}
Let $(M^n,g)$ be a smooth closed negatively curved Riemannian manifold of even dimension or dimension $7$. If the frame flow is not ergodic, then there exists a finite Riemannian cover $(\widehat{M},\widehat{g}$) such that the following holds. If $n \neq 7,8,134$, then:
\begin{itemize}
\item If $n=4$ or $n \equiv 2$ mod $4$, there exists an invariant unit vector field $f \in C^\infty(S\widehat{M},\mc{N})$ of odd degree,
\item If $n \equiv 0$ mod $4$, there exists an invariant orthogonal projector $f \in C^\infty(S\widehat{M},\Sym^2 \mc{N})$ of even degree such that $1 \leq \mathrm{rank}(f) \leq \min\big(\rho(n)-1,(n-2)/2\big)$.
\end{itemize}
In the three exceptional cases, we have:
\begin{itemize}
\item If $n=7$, there exists an invariant complex structure $f \in C^\infty(S\widehat{M},\Lambda^2 \mc{N})$ of odd degree,
\item If $n=8$, there exists an invariant $\mathrm{G}_2$-structure $f \in C^\infty(S\widehat{M},\Lambda^3 \mc{N})$ of odd degree or there exists an invariant orthogonal projector $f \in C^\infty(S\widehat{M},\Sym^2 \mc{N})$ of even degree such that $1 \leq \mathrm{rank}(f) \leq 3$,
\item If $n=134$, there exists a non-zero invariant Lie bracket\footnote{Here by Lie bracket we mean a $3$-form $\omega$ such that the bracket $[\cdot,\cdot]$ defined by $\langle [a,b],c\rangle:=\omega(a,b,c)$ satisfies the Jacobi identity.} $f \in C^\infty(S\widehat{M},\Lambda^3 \mc{N})$ of odd degree $\geq 3$ or there exists an invariant unit vector field $f \in C^\infty(S\widehat{M},\mc{N})$ of odd degree.
\end{itemize}
In all cases, the invariant elements have finite Fourier degree \cite{Guillarmou-Paternain-Salo-Uhlmann-16}.
\end{theorem}

We remark that in the case $n = 7$, by a complex structure we rather understand a section $f$ of $\Lambda^2 \mc{N}$ determined by an orthogonal complex structure $J(x, v): \mc{N}(x, v) \to \mc{N}(x, v)$ via the formula $f(x, v)(\cdot, \cdot) = g_x(J(x, v) \cdot, \cdot)$.

\begin{proof}
The proof is divided in two steps. First of all, we show the existence of these invariant structures and then we show that their Fourier degree has to be odd or even. 
 We will make use of the contraction operator $\imath_v:\pi^*TM^{\otimes p}\to \pi^*TM^{\otimes (p-1)}$ by the tautological section, defined on decomposable elements $K_{(x,v)}=v_1\otimes\ldots\otimes v_p$ by $$\imath_v(K)=g_x(v,v_1)v_2\otimes\ldots\otimes v_p.$$ For example, if $f \in C^\infty(SM,\pi^* \Lambda^p TM)$, then $\imath_v f$ is the usual contraction of forms and if $f \in C^\infty(SM,\pi^*\Sym^2 TM)$ is viewed at each point $(x,v)$ as a symmetric endomorphism of $T_xM$, then, $\imath_v f$ at $(x,v)$ is just the vector $f(v)$. From the definition it is clear that $\mc{N} \subset \pi^*TM$ is equal to $\pi^*TM \cap \ker \imath_v$, and more generally 
$\Lambda^p\mc{N}=\pi^* \Lambda^p TM\cap \ker \imath_v$ and $\Sym^2\mc{N}=\pi^* \Sym^2 TM\cap \ker \imath_v$.
 In particular, all sections $f$ obtained from Theorem \ref{theorem:reduction} satisfy $\imath_v f = 0$.
\\

\emph{Step 1: existence of invariant structures.} Let $(M^n,g)$ be a closed Riemannian manifold with negative sectional curvature and further assume that the frame flow is not ergodic. By Proposition \ref{proposition:brin}, the transitivity group $H$ is a proper subgroup of $\mathrm{SO}(n-1)$. If $H$ is not connected, we can apply Lemma \ref{lemma:brin} to produce a finite cover $(\widehat{M},\widehat{g})$ whose transitivity group of the frame flow is the identity component $H_0$ of $H$. Thus, without loss of generality, we can directly assume that $H$ is connected. This implies in particular that $(H,\iota)$ is a reduction of the structure group of the frame bundle $F\Ss^{n-1} \to \Ss^{n-1}$ over the sphere (which is a principal $\mathrm{SO}(n-1)$-bundle), where $\iota : H \to \mathrm{SO}(n-1)$ is the embedding\footnote{Note that, here, we see $H$ as a subgroup of $\mathrm{SO}(n-1)$, that is, equivalently, we ask the representation $\iota : H \to \mathrm{SO}(n-1)$ to be faithful.}. 

We claim that the following holds\footnote{This is stated and organized in a slightly different way on the level of groups in \cite[Proposition 3.1]{Cekic-Lefeuvre-Moroianu-Semmelmann-22}.}. If $n \neq 7,8,134$, then

\begin{itemize}
\item If $n$ is odd, there is no non-trivial reduction of the structure group of $F\mathbb{S}^{n - 1} \to \mathbb{S}^{n - 1}$,
\item If $n \equiv 2$ mod $4$ or $n=4$, $H$ fixes a vector of $\R^{n-1}$,
\item If $n \equiv 0$ mod $4$, $H$ acts reducibly on $\R^{n-1} = V \oplus V^\bot$. Moreover, assuming without loss of generality that $\dim V \leq \dim V^\perp$, we have $\dim V \leq \rho(n)-1$.
\end{itemize}
The following exceptional cases hold:
\begin{itemize}
\item If $n=7$, then $H = \mathrm{SU}(3)$ or $\mathrm{U}(3)$ and $H$ fixes an almost-complex structure structure in $\R^6$,
\item If $n=8$, then either: $H = \mathrm{SO}(3)$ or $\mathrm{G}_2$ and $H$ fixes a non-zero $3$-form in $\R^7$ (with $\mathrm{G}_2$ stabilizer), \emph{or} $H$ acts reducibly on $\R^7 = V \oplus V^\bot$ with $\dim V \leq 3$; 
\item If $n= 134$, then either: $H = \mathrm{E}_7/\Z_2$ and $H$ fixes a non-zero $3$-form on $\R^{133}$ (a Lie bracket), \emph{or} $H$ fixes a non-zero vector in $\R^{133}$.
\end{itemize}

We believe that the $\mathrm{E}_7/\Z_2$-structure on $\Ss^{133}$ actually never occurs and could be ruled out by purely topological arguments but we are unable to prove this. 

The claim mainly follows from \cite{Adams-62,Leonard-71,Cadek-Crabb-06} and requires only a few additional arguments which are given below. Taking this claim for granted, the proof of Theorem \ref{theorem:reduction} is then a straightforward consequence of Proposition \ref{proposition:invariant} which produces flow-invariant objects from vectors fixed by the representation, and parity properties shown in Step 2. When $H$ acts reducibly, the object considered is the orthogonal projector onto one of the summands (the one with the lowest dimension). \\

We now prove the claim. For $n$ odd, this follows from \cite[Theorem 1]{Leonard-71} and we assume from now on that $n$ is even. We further assume that $n \geq 10$, the low-dimensional cases are dealt afterwards. We look at the following dichotomy: \\

\textbf{A.} The action of $H$ on $\R^{n-1}$ is reducible, that is, $\R^{n-1} = V \oplus V^\bot$, where each summand is $H$-invariant. Up to changing the roles of $V$ and $V^\bot$, we can assume that $\dim V \leq \dim V^\bot$. Note that since $H$ is connected, it acts by orientation preserving isometries on both $V$ and $V^\perp$. By the proof of \cite[Theorem 2.A]{Leonard-71}, there must exist at least $\dim V$ vector fields on $\Ss^{n-1}$, that is $\dim V \leq \rho(n)-1$. Since $\rho(4\ell+2)-1=1$, we see that $H$ fixes a non-zero vector in $\mathbb{R}^{n-1}$ if $n \equiv 2$ mod $4$.\\

\textbf{B.} The action of $H$ on $\R^{n-1}$ is irreducible. Then $H$ has to be simple\footnote{Recall that a group $G$ is said to be simple if it is connected, non-Abelian, and every closed connected normal subgroup is either the identity or the whole group.} by \cite[Theorem 3]{Leonard-71}. We say that $H$ has type $\mathfrak{h}$, where $\mathfrak{h}$ is the Lie algebra of ${H}$.\\

\textbf{B.1.} Since an irreducible representation of $\mathfrak{so}(2)$ has dimension at most 2, $\mathfrak{so}(4)$ is not simple, and $\mathfrak{sp}(1)=\mathfrak{so}(3)=\mathfrak{su}(2)$, $H$ has type $\mathfrak{so}(k)$ with $k \geq 5$ or $\mathfrak{sp}(k), \mathfrak{su}(k)$ with 
$k \geq 2$.\\ 

 \textbf{B.1.1.} Let us first deal with the last two cases $\mathfrak{sp}(k)$ or $\mathfrak{su}(k)$. This implies that $H$ is a quotient of $\mathrm{Sp}(k)$ or $\mathrm{SU}(k)$ by a subgroup of its center. By Lemma \ref{lemma:lift}, we thus know that the reduction of the structure group $\iota : H \to \mathrm{SO}(n-1)$ lifts to a reduction $\widetilde{\iota} : \mathrm{Sp}(k), \mathrm{SU}(k) \to \mathrm{SO}(n-1)$. If $n \geq 10$, the existence of such an irreducible reduction $\widetilde{\iota}$ is impossible by \cite[Theorem 2.1 (B), (C)]{Cadek-Crabb-06} so $H$ cannot be of type $\mathfrak{sp}(k)$ or $\mathfrak{su}(k)$. \\

\textbf{B.1.2.}  We now deal with the case where $H$ has type $\mathfrak{so}(k)$. There are a few possibilities for $H$, namely it can be $\mathrm{Spin}(k),\ \mathrm{SO}(k)$, or additionally $\mathrm{PSO}(k)$ when $k$ is even, or the semi-spin group $\mathrm{SemiSpin}(k)$ if $k \equiv 0 \mod 4$.\\

\textbf{B.1.2.1.} If $H=\mathrm{PSO}(k)$, then there is also a reduction to $\mathrm{SO}(k)$ by Lemma \ref{lemma:lift}, in which case \cite[Theorem 2.1 (A)]{Cadek-Crabb-06} shows that $\mathrm{SO}(k) \hookrightarrow \mathrm{SO}(n-1)$ is the standard diagonal embedding so it cannot be irreducible. \\

\textbf{B.1.2.2.} If $H= \mathrm{SO}(k)$, the conclusion follows from the previous point. \\

\textbf{B.1.2.3.} We now deal with $\mathrm{Spin}(k)$ and $\mathrm{SemiSpin}(k)$ (when $k \equiv 0$ mod $4$). In the former case, it suffices to show that a group morphism $\iota : \mathrm{Spin}(k) \to \mathrm{SO}(n-1)$ cannot be faithful. Observe that $\iota(-1) = \pm 1$ since $-1$ is in the center of $\mathrm{Spin}(k)$ and $\iota(-1)^2=1$ (we use that $H$ is irreducible). If $\iota(-1)=1$, then we obtain that $\iota$ is not faithful and thus the representation factorizes through $\mathrm{SO}(k)$. Hence $\iota(-1)=-1$ but since $\e_1 \cdot \e_2 \in \mathrm{Spin}(k)$ squares to $-1$ (where $(\mathbf{e}_1, \ldots, \mathbf{e}_k)$ is an orthonormal basis of $\R^k$), we get that $\iota(\e_1\cdot\e_2)$ is a complex structure on $\R^{n-1}$ and this is a contradiction since $n-1$ is odd. The same argument works for $\mathrm{SemiSpin}(k)$. \\

\textbf{B.2.} $H$ has exceptional type, that is ${\mathfrak{h}}$ is isomorphic to one of the five exceptional simple Lie algebras and $H$ is isomorphic to (a finite quotient of) one of the exceptional Lie groups. So it suffices to look at real irreducible representations of exceptional Lie groups on odd dimensional vector spaces. Moreover, by \cite[Proposition 3.1]{Cadek-Crabb-06}, writing $n-1=\dim H + k + 1$ for some integer $k$, there must exist at least $k$ vector fields on the sphere $\Ss^{n-1}$, so the Radon-Hurwitz number satisfies 
\begin{equation}\label{rh1}\rho(n)\ge n-1-\dim H.\end{equation} 
For $n\ge 10$ we have $\rho(n)\leq\frac n2$, unless $n=16$. Since no exceptional Lie group has an irreducible representation in dimension $15$, it follows that $n-1$ has to be the dimension of an irreducible representation of an exceptional Lie group $H$, with 
\begin{equation}\label{rh}9\leq n-1\leq 2\dim H+1.\end{equation} 
The following possibilities may occur: \\

\textbf{B.2.1.} ${\mathfrak{h}}=\mathfrak{g}_2$, $\dim H=14$. The only odd-dimensional irreducible representation of $\mathfrak{g}_2$ satisfying \eqref{rh} has dimension $n-1=27$. However, $\rho(28)=4$, so it does not satisfy \eqref{rh1}.\\

\textbf{B.2.2.} ${\mathfrak{h}}=\mathfrak{f}_4$, $\dim H=52$. There is no odd-dimensional irreducible representation of $\mathfrak{f}_4$ satisfying \eqref{rh}.\\

\textbf{B.2.3.} ${\mathfrak{h}}=\mathfrak{e}_6$, $\dim H=78$. Again, there is no odd-dimensional irreducible representation of $\mathfrak{e}_6$ satisfying \eqref{rh}.\\

\textbf{B.2.4.} ${\mathfrak{h}}=\mathfrak{e}_7$, $\dim H=133$. The only odd-dimensional irreducible representation of $\mathfrak{e}_7$ satisfying \eqref{rh} is its adjoint representation. Note that the adjoint representation of $\mathrm{E}_7$ on $\mathbb{R}^{133}$ is not faithful since $\mathrm{E}_7$ has non-trivial center, so ${H}=\mathrm{E}_7/\Z_2$. In this case, $H$ fixes the Lie bracket on $\R^{133}=\mathfrak{e}_7$ which we can view as a vector in $\Lambda^3 \R^{133}$. \\

\textbf{B.2.5.} ${\mathfrak{h}}=\mathfrak{e}_8$, $\dim H=248$. There is no odd-dimensional irreducible representation of $\mathfrak{e}_8$ satisfying \eqref{rh}.\\

All relevant information about the low-dimensional irreducible representations of exceptional simple Lie algebras can be found in \cite{McKay-Patera-Rand-90}.

\medskip

\textbf{C.} Finally, it remains to deal with the low-dimensional cases $n=4, 6, 7, 8$. For $n=4$, the only strict connected subgroup of $\mathrm{SO}(3)$ is $\mathrm{SO}(2)$ and this fixes a vector in $\R^3$. For $n = 6$, $7$, or $8$, if the subgroup $H$ acts reducibly on $\R^{n-1}$, then it falls into case \textbf{A.} Hence we may assume $H$ acts irreducibly. It is straightforward to check that for $n=6,7,8$, the only groups acting irreducibly on $\mathbb{R}^{n-1}$ are $H = \mathrm{SO}(3)$ for $n = 6$, $H=\mathrm{SU}(3)$ or $\mathrm{U}(3)$ for $n=7$ and $H = \mathrm{SO}(3)$ or $\mathrm{G}_2$ for $n = 8$. (The embeddings of $\mathrm{SO}(3)$ in $\mathrm{SO}(5)$ and $\mathrm{SO}(7)$ are obtained via the irreducible representations of $\mathrm{SO}(3)$ in dimensions 5 and 7 respectively). The case $n=6$ and $H = \mathrm{SO}(3)$ is ruled out by \cite{Agricola-Becker-Bender-Friedrich-11}. If $n=7$, both $\mathrm{SU}(3)$ and $\mathrm{U}(3)$ fix the standard complex structure on $\R^6$, while for $n=8$, note that the irreducible $\mathrm{SO}(3)$ representation of dimension $7$ is the restriction of the irreducible 7-dimensional $\mathrm{G}_2$ representation, and so the group fixes a invariant vector in $\Lambda^3 \R^7$, the $\mathrm{G}_2$-structure.
\\

\emph{Step 2: topological reduction of the Fourier degree.} We now show that the degree of flow-invariant objects must be odd or even in certain cases, using topological arguments.

We start with the case of flow-invariant objects with values in $\Sym^2 \mc{N}$ and the following observation: if $f \in C^\infty(SM, \Sym^2 \mc{N})$ or $f \in C^\infty(SM,\Lambda^p \mc{N})$ is such that $\X f = 0$, then writing $f = f_{\mathrm{odd}} + f_{\mathrm{even}}$ where each term has respectively odd or even Fourier degree, we have $0 = \imath_v f = \imath_v f_{\mathrm{odd}} = \imath_v f_{\mathrm{even}}$ and $0 = \X f = \X f_{\mathrm{odd}} = \X f_{\mathrm{even}}$. (The proof is immediate as both operators $\X$ and $\imath_v$ shift the Fourier degrees by $\pm 1$.)

\begin{lemma}
Assume that $n$ is even and there exists a section $f \in C^\infty(SM,\Sym^2\mc{N})$ not proportional to the tautological section $\mathbbm{1}$, such that $\X f = 0$. Then, there exists a flow-invariant orthogonal projector $f' \in C^\infty(SM,\Sym^2\mc{N})$ of even degree such that $1\leq\mathrm{rank}(f') \leq \min\big(\rho(n)-1,(n-2)/2\big)$.
\end{lemma}

\begin{proof}
We write $f = f_{\mathrm{even}} + f_{\mathrm{odd}}$, where both terms are flow-invariant and take values in $\Sym^2 \mc{N}$. Each of them can be diagonalized and has constant (real) eigenvalues. We briefly sketch this (for details, see \cite[Lemma 5.6]{Cekic-Lefeuvre-20}): consider a dense orbit $\mc{O}$ of the geodesic flow, and an orthonormal frame for $\mc{N}|_{\mc{O}}$ obtained by parallel transport along the flow of an orthonormal frame of $\mc{N}(z)$ for some $z \in \mc{O}$. Using the flow-invariance, it follows that in this frame $f_{\mathrm{odd}/\mathrm{even}}$ are constant, and so have constant eigenvalues; by density of $\mc{O}$ we reach the same conclusion on the whole of $SM$. Further following \cite[Lemma 5.6]{Cekic-Lefeuvre-20}, one can show that $f_{\mathrm{even}} = \sum_i \lambda_i \Pi_{\lambda_i}$, where $\lambda_i$ are the constant eigenvalues and $\Pi_{\lambda_i}$ are the orthogonal projectors onto the corresponding eigenspaces. It can easily be checked that these projectors are all even. For the odd part of $f$, it can be shown that if $\lambda$ is an eigenvalue of $f_{\mathrm{odd}}$, then so is $-\lambda$ and $\Pi_\lambda(-v) = \Pi_{-\lambda}(v)$. Hence $f_{\mathrm{odd}} = \sum_i \lambda_i \Pi_{\lambda_i} - \lambda_i \Pi_{-\lambda_i}$ and $\Pi_{\lambda_i} + \Pi_{-\lambda_i}$ is an even orthogonal projection for each $i$.

In any case, we obtain a non-zero orthogonal projector $f'$ with even Fourier degree which is different from $0$ and  $\mathbbm{1}$. Let $F$ be the image of $f'$. Up to changing $f'$ by $\mathbbm{1}-f'$ (and $F$ by $F^\bot$), we can assume that $1\leq r:= \mathrm{rank}(f') \leq (n-2)/2$. Observe that $F$ is in particular a subbundle of the tangent bundle to the sphere and this also forces $r \leq \rho(n)-1$, see \cite{Leonard-71}.
\end{proof}

We now turn to the case where the flow-invariant section takes values in $\Lambda^p \mc{N}$. When $p=1$ we have the following:

\begin{lemma}
\label{lemma:topological-degree}
If the dimension $n$ of $M$ is even, then every section $f \in C^\infty(SM,\mc{N})$ which satisfies $\X f = 0$ has odd degree.
\end{lemma}

\begin{proof}
Write $f = f_{\mathrm{even}} + f_{\mathrm{odd}}$, the decomposition of $f$ into even and odd Fourier degrees. Both terms have constant norm on $SM$ and satisfy $\imath_v f_{\mathrm{even}} = \imath_v f_{\mathrm{odd}}=0$. Assume for a contradiction that $f_{\mathrm{even}} \neq 0$. Up to rescaling, we can assume its norm is constant equal to $1$. Fixing an arbitrary point $x_0 \in M$ and identifying $S_{x_0}M \simeq \Ss^{n-1}$, the section $f_{\mathrm{even}}$ defines a non-vanishing vector field $Z(v) := f(x_0,v)$ such that $Z \in C^\infty(\Ss^{n-1}, T\Ss^{n-1})$ satisfying $Z(v) = Z(-v)$ by evenness. We can see $Z$ as a smooth map $Z : \Ss^{n-1} \to \Ss^{n-1}$. Now, $Z$ is clearly homotopic to the identity by
\[
Z_t(v) := \cos\left(\tfrac{\pi t}{2}\right) v + \sin\left(\tfrac{\pi t}{2}\right) Z(v),
\]
hence its topological degree is $1$. On the other hand, any regular value of $Z$ has an even number of preimages since $Z(v) = Z(-v)$, so the topological degree of $Z$ is even. This is a contradiction.
\end{proof}

Consider now the case $p\ge 2$. We will obtain a general result about the Fourier degree of $p$-forms on the sphere $\Ss^{n-1}$ which are defined by reductions of the structure group. Assume that the $\SO(n-1)$-principal bundle $\SO(n)\to \Ss^{n-1}$ ($=\SO(n)/\SO(n-1)$) of oriented orthonormal frames on the sphere $\Ss^{n-1}$ has a reduction to $H\subset \SO(n-1)$ and that the following properties hold:
\begin{enumerate}
\item \label{enum:1} for some $p\ge 2$, the restriction to $H$ of the standard $\SO(n-1)$-representation on $\Lambda^p\R^{n-1}$ has an invariant vector $\omega_0$;
\item \label{enum:2} the stabilizer of $\omega_0$ in $\Oo(n-1)$ is contained in $\SO(n-1)$.
\end{enumerate}

Note that the property \eqref{enum:2} implies that $\omega_0$ is non-zero. Equivalently to Definition \ref{definition:structure}, the assumption that the structure group of $\Ss^{n-1}$ reduces to $H$ means that there exists an $H$-principal subbundle $\SO(n)\supset P\to \Ss^{n-1}$. The associated vector bundle $E:=P\times_H\Lambda^p\R^{n-1}$ is canonically isomorphic to the exterior bundle $\Lambda^pT\Ss^{n-1}$ and for every $v\in \Ss^{n-1}$ and $u$ in the fibre $P_v$ of $P$ over $v$, one can view $u$ as an isomorphism$u:\Lambda^p\R^{n-1}\to \Lambda^p_vT\Ss^{n-1}$.

Using property \eqref{enum:1}, one can define a $p$-form $\omega\in\Omega^p(\Ss^{n-1})$ by 
$$\omega_v:=u(\omega_0),\qquad\forall v\in \Ss^{n-1},\ \forall u\in P_v$$ 
(the definition is clearly independent on the choice of $u$ in the fiber of $P$ over $v$). Our aim is to prove the following:

\begin{lemma} \label{l3.9}The $p$-form $\omega$ on $\Ss^{n-1}$ cannot have even Fourier degree.
\end{lemma}

\begin{proof} Assume for a contradiction that $\omega$ has even Fourier degree. In particular, for every $v\in \Ss^{n-1}$, if we identify the tangent spaces at $v$ and $-v$ to the sphere, and correspondingly their exterior powers, we have 
\be\label{o}\omega_v=\omega_{-v}.\ee 

Take some arbitrary frames $u\in P_v$ and $u'\in P_{-v}$. Note that $u:\R^{n-1}\to T_v\Ss^{n-1}$ and $u':\R^{n-1}\to T_{-v}\Ss^{n-1}$ are direct isometries, but  $T_v\Ss^{n-1}$ and $T_{-v}\Ss^{n-1}$ have opposite orientations. This means that $a:=u^{-1}\circ u'$ is an isometry of $\R^{n-1}$ reversing the orientation.  By \eqref{o}, we get
$$a(\omega_0)=(u^{-1}\circ u')(\omega_0)=u^{-1}(\omega_{-v})=u^{-1}(\omega_{v})=\omega_0,$$
thus contradicting property (2).
\end{proof}

We can now treat the remaining exceptional cases. 

\begin{lemma}\label{3.13}
Assume that the transitivity group acts irreducibly on $\R^{n-1}$, and either
\begin{itemize}
\item  $n=7$ and $p = 2$;
\item or $n=8$ and $p = 3$; 
\item or $n= 134$ and $p = 3$.
\end{itemize}
Then every non-zero flow-invariant section $f \in C^\infty(SM,\Lambda^p \mc{N})$ has odd Fourier degree.
\end{lemma}
\begin{proof}
By {\bf B.2.3} and {\bf C} in Step 1 above, we have that $H=\mathrm{SU}(3)$ or $\mathrm{U}(3)$ for $n=7$, $H = \mathrm{SO}(3)$ or $\mathrm{G}_2$ for $n = 8$ and $H= \mathrm{E}_7/\Z_2$ for $n=134$.
One can easily check\footnote{Using the LiE program for instance.} that in all cases the $H$-representation on $\Lambda^p\R^{n-1}$ has exactly one trivial summand, generated by a $p$-form satisfying the properties \eqref{enum:1} and \eqref{enum:2} above. 

More precisely, denoting by $(\e_1,\e_2,\ldots,\e_n)$ an orthonormal basis in $\R^n$ and setting $\e^{i_1\ldots i_k} := \e_{i_1} \wedge \ldots \wedge \e_{i_k}$ for $1 \leq i_1 < \ldots < i_k \leq n$, the following holds:

\begin{enumerate}[itemsep=5pt]
\item[a)] The stabilizer of the $2$-form $$\omega_0:=\e^{14}+\e^{25}+\e^{36}$$ in $\Oo(6)$ is $\mathrm{U}(3)\subset\SO(6)$, and $\R\omega_0$ is the unique trivial summand of the representations of $\mathrm{SU}(3)$ and $\mathrm{U}(3)$ on $\Lambda^2\R^6$. (Here one can view the $2$-form $\omega_0$ as the complex structure on $\mathbb{R}^6$.)

\item[b)] The stabilizer of the $3$-form
\[\omega_0 = \e^{123} + \e^{145} + \e^{167} + \e^{246} - \e^{257} - \e^{347} - \e^{356}\]
in $\Oo(7)$ is $\mathrm{G}_2\subset\SO(7)$ (see \cite[Proof of Theorem 1]{Bryant1987}), and $\R\omega_0$ is the unique trivial summand of the representations of $\mathrm{SO}(3)$ and $\mathrm{G}_2$ on $\Lambda^3\R^7$.

\item[c)]We identify $\R^{133}$ with the Lie algebra $\mathfrak{e}_7$ and consider the 3-form $\omega_0\in\Lambda^3\mathfrak{e}_7$ defined by
$$\omega_0(x,y,z):=\langle [x,y],z\rangle,\qquad\forall x,y,z\in\mathfrak{e}_7.$$
For $a\in \Oo(133)$, we have $a(\omega_0)=\omega_0$ if and only if $a[x,y]=[ax,ay]$ for every $x,y\in\mathfrak{e}_7$, i.e. if and only if $a\in \mathrm{Aut}(\mathfrak{e}_7)$. Since all automorphisms of $\mathfrak{e}_7$ are inner (due to the fact that its Dynkin diagram has no symmetries), we deduce that the stabilizer of $\omega_0$ in $\Oo(133)$ is equal to $\mathrm{E}_7/\mathbb{Z}_2 \subset \SO(133)$. Again, $\R\omega_0$ is the unique trivial summand of the representations of $\mathrm{E}_7/\mathbb{Z}_2$ on $\Lambda^3\R^{133}$.
\end{enumerate}

Let $ f =  f_{\mathrm{even}} +  f_{\mathrm{odd}}$ be the flow-invariant section obtained by Proposition \ref{proposition:invariant}. Observe that both $f_{\mathrm{even}}$ and $f_{\mathrm{odd}}$ are flow-invariant, thus fixed by $H$. By the fact that the $H$-representation on $\Lambda^p\R^{n-1}$ has only one trivial summand, we have either $f=f_{\mathrm{even}}$ or $f=f_{\mathrm{odd}}$. Since $f$ satisfies the properties \eqref{enum:1} and \eqref{enum:2} above, Lemma \ref{l3.9} shows that $f$ cannot be even. We thus have $f= f_{\mathrm{odd}} \neq 0$.
\end{proof}

 We finally prove the following result.

\begin{lemma}\label{e7}
Assume that $n=134$ and $H$ is a subgroup of $\mathrm{E}_7/\Z_2$. Let $f \in C^\infty(SM,\Lambda^3 \mc{N})$ be the flow-invariant Lie bracket obtained by Proposition \ref{proposition:invariant}, i.e. a section which is equivalent at every point $(x,v)$ to the canonical $3$-form of the Lie algebra $\mathfrak{e}_7$, and such that $\X f = 0$. Then $f$ has degree at least $3$.
\end{lemma}

\begin{proof}
By Lemma \ref{3.13}, $f$ had odd Fourier degree. Assume for the sake of a contradiction that the degree of $f$ is $1$. Then $f = f_1 \in C^\infty(M,TM \otimes \Lambda^3 TM)$. The condition $\imath_v f = 0$ reads $f(v,v,\cdot,\cdot)=0$, that is $f$ is a $4$-form on $M$ which we denote by $\phi \in C^\infty(M,\Lambda^4 TM)$.

The hypothesis on $f$ shows that for every unit vector $v$, there exists an isometry between $v^\perp$ and $\mathfrak{e}_7$, such that the $3$-form $\iota_v\phi$ is the pull-back of the canonical $3$-form of $\mathfrak{e}_7$. We will identify them by a slight abuse of notation. For every $u\in v^\perp$ one can thus interpret the $2$-form $\phi(v,u)$ as corresponding to the endomorphism $\mathrm{ad}_u$ acting on $\mathfrak{e}_7$. By the irreducibility of the adjoint representation of $\mathrm{E}_7$ we have $|\mathrm{ad}_u|^2=c|u|^2$ for some positive constant $c$ which only depends on the structure of $\mathrm{E}_7$ (indeed, $\mathfrak{e}_7 \ni u \mapsto |\mathrm{ad}_u|^2$ is an $\mathrm{E}_7$-invariant quadratic form so by Schur's lemma, it is proportional to the metric).

We can write this as
\begin{equation}
\label{equation:iso-weak}
|\phi(u,v)|^2=c|u\wedge v|^2,
\end{equation}
for every $u,v\in TM$. We can also see $\phi$ as a symmetric endomorphism $\phi : \Lambda^2 TM \to \Lambda^2 TM$ (the symmetry comes from the relation $\phi(u,v,w,z) = \langle \phi(u \wedge v), w \wedge z \rangle = \phi(w,z,u,v) = \langle \phi(w,z),u \wedge v\rangle$) and \eqref{equation:iso-weak} says that $\phi$ is an isometry on decomposable elements.

 By polarization in $u$ we obtain $\langle\phi(u,v),\phi(w,v)\rangle=c\langle u\wedge v,w\wedge v\rangle$, and by polarization in $v$ we get:
\begin{equation}
\label{equation:pola}
\langle\phi(u,z),\phi(w,v)\rangle + \langle\phi(u,v),\phi(w,z)\rangle=c\left(\langle u\wedge z,w\wedge v\rangle + \langle u\wedge v,w\wedge z\rangle \right).
\end{equation}
We now fix a unit vector $a \in TM$ and consider $u,v,w,z \in a^\perp$. The Jacobi identity on $a^\perp \simeq \mathfrak{e}_7$ implies that the following cyclic sum vanishes:
\[
\mathfrak{S}_{u,v,w} \langle \mathrm{ad}_v w, \mathrm{ad}_u z \rangle = 0 = \mathfrak{S}_{u,v,w} \langle \phi(a,v,w),\phi(a,u,z)\rangle = \mathfrak{S}_{u,v,w} \langle \phi(w,v,a),\phi(u,z,a)\rangle.
\]
This identity also holds for $u,v,w,z \in TM$ since $\phi(a,a,\cdot,\cdot)=0$ so it holds for all $a,u,v,w,z \in TM$. Taking the trace in $a$ over an orthonormal basis $(\e_1,\ldots,\e_n)$ of $TM$, we get:
\[
0 =  \mathfrak{S}_{u,v,w} \langle \phi(w,v),\phi(u,z)\rangle,
\]
that is
\begin{equation}
\label{equation:jacobi}
\langle \phi(u,z),\phi(w,v)\rangle + \langle \phi(w,z),\phi(v,u) \rangle + \langle \phi(v,z),\phi(u,w) \rangle = 0.
\end{equation}
Similarly, we have
\begin{equation}
\label{equation:jacobi2}
\langle u \wedge z, w \wedge v \rangle + \langle w \wedge z, v \wedge u\rangle + \langle v \wedge z, u \wedge w \rangle = 0.
\end{equation}
As a consequence, setting $F := \phi^2 - c \mathbbm{1}_{\Lambda^2}$, we get that $F$ is a symmetric endomorphism and using \eqref{equation:pola}, \eqref{equation:jacobi} and \eqref{equation:jacobi2} it satisfies the relations:
\[
\begin{array}{l}
\langle F(u \wedge z), w \wedge v \rangle + \langle F(u \wedge v), w \wedge z \rangle = 0,\\
\langle F(u \wedge z),w\wedge v\rangle + \langle F(w \wedge z), v \wedge u \rangle + \langle F(v \wedge z), u \wedge w\rangle = 0.
\end{array}
\]
It is straightforward to check that these relations imply that $F = 0$. Thus, setting $\phi' := \frac{1}{\sqrt{c}} \phi$, we get that $\phi'$ is symmetric and $\phi'^2 = \mathbbm{1}_{\Lambda^2}$, that is $\phi'$ is an orthogonal symmetry. Thus the trace of $\phi'$ is equal to the difference of the dimensions of its $1$ and $-1$ eigenspaces, so it is an odd integer since $\Lambda^2(TM)$ has odd dimension $\frac{134\times 133}{2} = 67 \times 133$. However, this is a contradiction since the trace of $\phi'$ is
$$\Tr(\phi')=\sum_{i<j}\langle \phi'(\e_i\wedge \e_j),\e_i\wedge \e_j\rangle=\frac{1}{\sqrt{c}} \sum_{i<j}\phi(\e_i,\e_j,\e_i,\e_j)=0.$$
\end{proof}

We observe that the above result actually holds more generally by replacing $\mathrm{E}_7/\mathbb{Z}_2$ with any simple compact Lie group of dimension $n-1=4k+1$.

The proof of Theorem \ref{theorem:reduction} is now complete.
\end{proof}

\begin{remark}
As a concluding remark, we observe that the ergodicity of the frame flow on negatively curved manifolds of odd dimension $n \neq 7$ proved by Brin-Gromov \cite{Brin-Gromov-80} is actually an immediate consequence of \cite[Theorem 1.A]{Leonard-71} which shows that there is no reduction of the structure group of $F \Ss^{n-1} \to \Ss^{n-1}$, unless $n=7$ and $H=\mathrm{U}(3)$ or $\mathrm{SU}(3)$.
\end{remark}

\section{Non-existence of flow-invariant structures under pinching conditions}

\label{section:invariant}

Theorem \ref{theorem:reduction} shows that non-ergodicity of the frame flow implies the existence of flow-invariant structures on $S\widehat{M}$, where $(\widehat{M}, \widehat{g}) \to (M, g)$ is a finite Riemannian cover. We now show that this is impossible under some pinching conditions.

\begin{theorem}
\label{theorem:invariant-structures}
Let $(M^n,g)$ be a closed Riemannian manifold with $\delta$-pinched negative sectional curvature. Then:
\begin{itemize}
\item There exists $\delta_{\Lambda^1}(n)$ given in \eqref{equation:delta-n-1} such that if $\delta > \delta_{\Lambda^1}(n)$, then there are no non-trivial odd flow-invariant sections $f \in C^\infty(SM,\mc{N})$,
\item There exists $\delta_{\Sym^2}(n)$ given in \eqref{equation:delta-sym-2} such that if $\delta > \delta_{\Sym^2}(n)$, then there are no non-trivial even flow-invariant orthogonal projectors $f \in C^\infty(SM,\Sym^2 \mc{N})$ of rank $r \leq \min\big(\rho(n)-1,(n-1)/2\big)$.
\end{itemize}
Moreover, the following exceptional cases hold:
\begin{itemize}
\item If $n=7$ and $\delta > \delta_{\mathrm{U}(3)}(7) := 0.4962...$, then there are no non-trivial odd flow-invariant almost-complex structures $f \in C^\infty(SM,\Lambda^2 \mc{N})$,
\item If $n=8$ and $\delta > \delta_{\mathrm{G}_2}(8) := 0.6212...$, then there are no non-trivial odd flow-invariant $\mathrm{G}_2$-structures $f \in C^\infty(SM,\Lambda^3 \mc{N})$,
\item If $n=134$ and $\delta > \delta_{\mathrm{E}_7}(134) := 0.5788...$, then there are no non-trivial odd flow-invariant Lie brackets $f \in C^\infty(SM,\Lambda^3 \mc{N})$ of degree $\geq 3$.
\end{itemize}
\end{theorem}

The expressions for the thresholds are given by:
\begin{equation}
\label{equation:delta-n-1}
 \delta_{\Lambda^1}(n) = \left\{ \begin{array}{ll}  \tfrac{\tfrac{2}{3}\sqrt{3(n^2 - 1)} + \tfrac{1}{2}(n + 3)}{3(n + 1) + \tfrac{2}{3}\sqrt{3(n^2 - 1)} - \tfrac{1}{2} + \tfrac{1}{2} \tfrac{(n + 2)(5n + 2)}{n + 4}}, & \text{ if }~ n \leq 8, \\ 
\tfrac{\tfrac{2}{3} \sqrt{3(n^2 - 1)} + \tfrac{1}{2}}{3(n + 1) + \tfrac{2}{3}\sqrt{3(n^2 - 1)} - \tfrac{1}{2}},& \text{ if } n \geq 10,
\end{array} \right.
 \end{equation}
and:
\begin{equation}
\label{equation:delta-sym-2}
\delta_{\Sym^2}(n) = \tfrac{n+5+\tfrac{8}{3}\sqrt{(n-1)(n+2)}+\tfrac{2(n+2)(n+4)}{3(n+1)(n+6)}\left( n+3+\tfrac{4}{3} \sqrt{3(n^2-1)}\right)}{3(n+1) + \tfrac{8}{3}\sqrt{(n-1)(n+2)} + \tfrac{2(n+2)(n+4)}{3(n+1)(n+6)} \left(5n+3+\tfrac{4}{3}\sqrt{3(n^2-1)} \right)}.\end{equation}
The first values for $\delta_{\Lambda^1}$ are $\delta_{\Lambda^1}(4)= 0.2928...,\delta_{\Lambda^1}(6)= 0.2823...$ and appear in Theorem \ref{theorem:ergodicity}. The combination of Theorem \ref{theorem:reduction} and Theorem \ref{theorem:invariant-structures} immediately proves Theorem \ref{theorem:ergodicity}. 
The remaining part of the paper is devoted to the proof of Theorem \ref{theorem:invariant-structures}.

\label{section:degree}

\subsection{Normal twisted conformal Killing tensors}

In the following, we let $\E = \Lambda^p TM$ or $\E = \Sym^2 TM$. We have proved in Theorem \ref{theorem:reduction} that non-ergodicity of the frame flow gives rise to a flow-invariant smooth section $f \in C^\infty(SM,\pi^*\E)$ such that $\X f = 0$ and $\imath_v f = 0$. By \cite{Guillarmou-Paternain-Salo-Uhlmann-16}, such an invariant section must have finite Fourier content, i.e. it can be written as $f = f_0 + f_1 + \dotso + f_k$, where $f_i \in C^\infty(M,\Omega_i \otimes \E)$ and $f_k \neq 0$. We now define $u := f_k$. The two conditions $\X f = 0$ and $\imath_v f = 0$ then translate into the fact that $\X_+ u = 0$ and $\imath_v u \in C^\infty(M, \Omega_{k-1} \otimes \E)$, i.e. $\imath_v u$ is two degrees less than expected. Alternatively, we can see $u$ as the pullback to $SM$ of a twisted conformal Killing tensor on the base, i.e. there exists $K \in C^\infty(M,\Sym^k_0 TM \otimes \E)$ such that $u = \pi_k^* K$, $\mc{P} D K = 0$ (conformal Killing condition) and satisfying also the algebraic condition $\imath_v \pi_k^*K$ is of degree $k-1$. In the case where $\E = \Lambda^p TM$, we also observe that $\imath_v \imath_v u = 0$ since $u$ is a form. In the case where $\E = \Sym^2 TM$ and $f$ is an even orthogonal projector, we have:

\begin{lemma}
\label{lemma:algebra}
Let $f \in C^\infty(SM, \Sym^2 \mc{N})$ be an orthogonal projector of rank $r \geq 1$ with (finite) even degree $k$ and let $u:=f_k$. Then $k \geq 2$, $\imath_v u$ is of degree $k-1$ and $\imath_v \imath_v u$ is of degree $k-2$. Moreover, if $k=2$, then we have $f = \frac{r}{n} \mathbbm{1}_{TM} + f_2$, where $f_2 \in C^\infty(M, \Omega_2 \otimes \Sym^2 TM)$.
\end{lemma}

\begin{proof}
If $k=0$, $f = f_0$ can be identified with an element in $C^\infty(M,\Sym^2 TM)$ and the condition $\iota_v f = 0$ implies $f=0$.

We write $f = \pi_k^* K$ for some tensor $K \in C^\infty(M,\Sym^k TM \otimes \Sym^2 TM)$. The condition $\imath_v f = 0$ is the same as $K(v,\ldots,v,v,\cdot) = 0$. Differentiating (on $SM$) in $v$ and taking $w = \partial_v$, using the symmetries of the tensor, we get $kK(w,v,\ldots,v,v,\cdot)+K(v,\ldots,v,w,\cdot)=0$. Applying to $v$ and $w$, this gives the relations
\begin{equation}
\label{equation:relations1}
k K(w,v,\ldots,v,v)+ K(v,\ldots,v,w)=0, ~~~~ k K(w,v,\ldots,v,v,w) + K(v,\ldots,v,w,w)=0.
\end{equation}
Differentiating once again the first of these relations, we get
\begin{equation}
\label{equation:relations2}
3k K(w,v,\ldots,v,w) + k(k-1) K(w,w,v,\ldots,v,v) + K(v,\ldots,v,w,w) = 0.
\end{equation}
Combining the second relation of \eqref{equation:relations1} and \eqref{equation:relations2}, we get
\begin{equation}
\label{equation:relation}
K(v,\ldots,v,w,w) = \frac{k(k-1)}{2} K(w,w,v,\ldots,v).
\end{equation}
By \eqref{equation:decomposition}, we can write $K = \sum_{i=0}^{k/2} \mc{J}^{k-2i} K_{2i}$, where $K_{2i} \in \Sym^{2i}_0 TM \otimes \Sym^2 TM$ and we have $f = \pi_k^* K = \sum_{i=0}^{k/2} f_{2i}$ with $f_{2i} = \pi_{2i}^* K_{2i}$. Taking the trace in the $w$-variable in \eqref{equation:relation}, we then obtain:
\[
\begin{split}
r & = \Tr(f(v))=\sum_{i=1}^n K(v,\ldots,v,\e_i,\e_i) \\
& = \frac{k(k - 1)}{2} \sum_{i=1}^n K(\e_i,\e_i,v,\ldots,v) 
= \frac{k(k - 1)}{2} (\mc{I} K)(v,\ldots,v) \\
& = \frac{k(k - 1)}{2} \sum_{i=0}^{k/2-1} c_i \mc{J}^{k-2-2i} K_{2i}(v,\ldots,v) = \frac{k(k - 
1)}{2} \sum_{i=0}^{k/2-1} c_i  \imath_v \imath_v f_{2i}(v)\end{split}
\]
where $c_i > 0$ are some positive constants. The term of highest degree in the last sum is $\imath_v \imath_v f_{k-2}(v)$ which, in principle, could be a sum of spherical harmonics of degrees $k-4$, $k-2$ and $k$. But since the total sum is equal to $r$, the highest degree vanishes so each term in the sum has degree $\leq k-2$. Using the relation $\imath_v \imath_v f = 0 = \imath_v \imath_v u + \sum_{i=0}^{k/2-1} \imath_v \imath_v f_{2i}$, we then deduce that $\imath_v \imath_v u$ has degree $k-2$.

In the particular case where $k=2$, taking unit $v$ and $w$, \eqref{equation:relation} gives that
\[
\imath_w \imath_w (f(v)) = \imath_v \imath_v (f_0 + f_2(w)),
\]
and taking the trace in $w$, we obtain $\Tr(f(v))=r= n \langle f_0 v, v\rangle$ for all $v$, that is $f_0 = \frac{r}{n} \mathbbm{1}_{TM}$. 
\end{proof}

 It is now worth introducing the following terminology:

\begin{definition}
\label{definition:nckt}
We say that $K \in C^\infty(M,\Sym^k_0 TM \otimes \E)$ is a \emph{normal twisted conformal Killing tensor}, if it satisfies $\mc{P} D K = 0$ (twisted conformal Killing condition), $\imath_v \pi_k^* K$ is of degree $k-1$ and $\imath_v \imath_v \pi_k^* K$ is of degree $k-2$ (normal condition).
\end{definition}

For $\E = \Lambda^p TM$, the condition on $\imath_v \imath_v \pi_k^*K$ is automatically satisfied since it is always $0$, while for $\E = \Sym^2 TM$, in our situation it is guaranteed by Lemma \ref{lemma:algebra}.
The proof of Theorem \ref{theorem:invariant-structures} will be a consequence of Proposition \ref{proposition:calcul} below. First of all, we introduce the constants:
\begin{equation}
\label{equation:constant-b}
\begin{array}{l}
B^{\Lambda^p}_{n,k,\delta} :=  \delta k(n+k-2) -  \dfrac{(1+\delta)p}{2}  -  \dfrac{2p}{3}(1-\delta) \left[k(n+k-2)(n-1)\right]^{1/2}, \\B^{\Sym^2}_{n,k,\delta} := B^{\Lambda^2}_{n,k,\delta},
\end{array}
\end{equation}
 for $k,p\geq 0$, and
\begin{equation}
\label{equation:constant-c}
\begin{array}{l}
C^{\Lambda^p}_{n,k,\delta} :=\dfrac{k(n+k-2)(n+2k-4)}{(n+k-3)(k-1)(n+2k-2)} B^{\Lambda^{p-1}}_{n,k-1,\delta} - \dfrac{ (n+2k-4)(1+\delta)}{2}, \\
C^{\Sym^2}_{n,k,\delta} :=\dfrac{k(n+k-2)(n+2k-4)}{(n+k-3)(k-1)(n+2k-2)} B^{\Lambda^{1}}_{n,k-1,\delta} - (n+2k-4)(1+\delta), \\
\hspace{4cm} 
\end{array}
\end{equation}
 for $k\geq 2$, $p\geq 1$. For $k=1$, we define $C^{\Lambda^p}_{n,1,\delta} = -(n-2)\frac{1+\delta}{2}$ and $C^{\Sym^2}_{n,1,\delta} = -(n-2)(1+\delta)$. We define
 \[
 D_{n,k,\delta} := \dfrac{(n+2k-6)k(n+k-2)(n+2k-4)}{(n+k-3)(k-1)(n+2k-2)}\dfrac{(1+\delta)}{2},
 \]
 with the convention that this is $0$ for $k=1$.

\begin{proposition}
\label{proposition:calcul}
Let $K \in C^\infty(M,\Sym^k_0 TM \otimes \E)$ be a normal twisted conformal Killing tensor and further assume that $k \geq 2$ if $\E = \Lambda^p TM$ and $k \geq 3$ if $\E=\Sym^2 TM$. If $u:=\pi_k^* K \in C^\infty(M, \Omega_k \otimes \E)$ denotes the corresponding section of $\pi^*\E$ over $SM$, then we have:
\begin{equation}
\label{equation:contradiction}
B^{\E}_{n,k,\delta} \|u\|^2 + C^{\E}_{n,k,\delta} \|\imath_v u\|^2 - D_{n,k,\delta}\|\imath_v \imath_v u\|^2 \leq 0.
\end{equation}
\end{proposition}

\begin{proof}
The normal twisted conformal Killing condition reads:
\begin{equation}
\label{equation:assumptions}
\X_+ u = 0,\,\,\, \imath_v u\,\, \mathrm{has\,\,degree}\,\, k - 1,\,\,\, \imath_v \imath_v u\,\, \mathrm{has\,\,degree}\,\, k - 2.
\end{equation}
We consider $u \in C^\infty(M, \Omega_k \otimes \E)$ satisfying \eqref{equation:assumptions}. The $\X_+u$ term in the localized Pestov identity \eqref{equation:local-pestov} vanishes and we get:
\begin{equation}
\label{equation:central-equality}
\dfrac{(n+k-2)(n+2k-4)}{n+k-3} \|\X_-u\|^2 + \|Z(u)\|^2  =  \langle R\nabla_{\V}^{\E}u, \nabla_{\V}^{\E}u \rangle +  \langle \mc{F}^{\E}u, \nabla_{\V}^{\E}u \rangle.
\end{equation}
We will bound the terms on the left-hand side of \eqref{equation:central-equality} from below while we will bound the terms on the right-hand side from above. The term $Z(u)$ seems difficult to control and we simply use $\|Z(u)\|^2 \geq 0$. The first term on the right-hand side of \eqref{equation:central-equality} involves the curvature tensor $R^{\E}$ can be decomposed according to \S \ref{ssection:curvature-tensor-bounds} as $R^{\E} = R_0^{\E} + \frac{1 + \delta}{2}G^{\E}$. 
Using \eqref{eq:F^E}, we write correspondingly for any orthonormal basis $(\e_i)_{i = 1}^n \subset T_xM$
\begin{equation}
\label{equation:splitting}
\mc{F}^{\E} = \mc{F}_0^{\E} + \frac{1 + \delta}{2}.\mc{G}^{\E}, \quad \mc{F}^{\E}_0(x, v) = \sum_{i = 1}^n \e_i \otimes R_0^{\E}(v, \e_i), \quad \mc{G}^{\E}(x, v) = \sum_{i = 1}^n \e_i \otimes G^{\E}(v, \e_i).
\end{equation}
We first deal with the term $\mc{F}^{\E}_0$:

\begin{lemma}
\label{lemma:inequality-1}
Given $u \in C^\infty(M,\Omega_k \otimes \E)$, one has:
\[
| \langle \mc{F}_0^{\E}u, \nabla_{\V}^{\E}u \rangle | \leq \dfrac{2p}{3}(1-\delta) \left[k(n+k-2)(n-1)\right]^{1/2} \|u\|^2,
\]
with the convention that $p=2$ if $\E=\Sym^2 TM$.
\end{lemma}

\begin{proof}
We fix $x \in M$. Below, all the scalar products below are the ones on $L^2(S_xM)$, $(\e_i)_{1 \leq i \leq n}$ is an orthonormal basis of $T_xM$ and $(\e_\alpha)_\alpha$ is an orthonormal basis of $\E$. We have, writing $\nabla_{\V} u_\alpha = \sum_i \langle \nabla_{\V}u_\alpha, \e_i \rangle \e_i$ that:
\[
\begin{split}
\langle \mc{F}_0^{\E}u, \nabla_{\V}^{\E}u \rangle_{L^2(S_xM)} & = \sum_{\alpha} \int_{S_xM} \langle R_0^{\E}(v,\nabla_{\V} u_\alpha) u,\e_\alpha\rangle dv \\
& = \sum_i \int_{S_xM}\langle R_0^{\E}(v,\e_i-v_i.v)u ,\sum_\alpha \langle \nabla_{\V} u_\alpha,\e_i\rangle \e_\alpha\rangle dv.
\end{split}
\]
By \eqref{equation:rolambdap} and \eqref{equation:rosym2}, we have:
\[
|\langle R_0^{\E}(X,Y)\omega,\tau\rangle|\leq \tfrac{2p}{3}(1-\delta) |X||Y||\omega||\tau|,
\]
for every tangent vector $X$, $Y$ and $p$-forms or symmetric $2$-tensors $\omega$, $\tau$ (where $p=2$ in the latter case). Hence, by Cauchy-Schwarz:
\begin{equation}\label{eq:Cauchy-Schwarz}
\begin{split}
& |\langle \mc{F}_0^{\E}u, \nabla_{\V}^{\E}u \rangle_{L^2(S_xM)}| \\
& \leq \frac{2p}{3}(1-\delta) \sum_i \int_{S_xM} |\e_i-v_i.v| |u| \left|\sum_\alpha \langle \nabla_{\V} u_\alpha,e_i\rangle \e_\alpha \right| dv \\
& \leq \frac{2p}{3}(1-\delta) \left( \sum_i \int_{S_xM} |\e_i-v_i.v|^2 |u|^2 dv \right)^{1/2} \left( \sum_i \int_{S_xM} \left|\sum_\alpha \langle \nabla_{\V} u_\alpha,\e_i\rangle \e_\alpha \right|^2 dv \right)^{1/2} \\
& = \frac{2p}{3}(1-\delta) \left( \int_{S_xM} (n-1) |u|^2 dv \right)^{1/2} \left( \sum_i \int_{S_xM} \sum_\alpha \langle \nabla_{\V} u_\alpha,\e_i\rangle^2 dv \right)^{1/2} \\
& =\frac{2p}{3}(1-\delta) (n-1)^{1/2} \|u\| \left( \sum_i \int_{S_xM} \sum_\alpha \langle \nabla_{\V} u_\alpha,\e_i\rangle^2 dv \right)^{1/2}.
\end{split}
\end{equation}
Observe that we may compute fibre-wise in $T_xM$, where $\partial_i$ denotes differentiation in $\e_i$:
\begin{equation}\label{eq:vertical_deriv_i}
\langle \nabla_{\V} u_\alpha,\e_i \rangle = \langle \nabla_{\V}^{\mathrm{tot}} u_\alpha - v.\langle{\nabla_{\V}^{\mathrm{tot}}u_\alpha, v}\rangle,\e_i\rangle = \partial_i u_\alpha - v_i.ku_\alpha,
\end{equation}
where $\nabla_{\V}^{\mathrm{tot}}(\cdot) = \nabla_{\V}(\cdot) + v. \langle{\cdot, v}\rangle$ is the total gradient, and the last equality follows from Euler's relation on homogeneous functions. Using the same relations, we obtain:
\[
\begin{split}
 \sum_{\alpha,i} \langle \nabla_{\V} u_\alpha ,\e_i\rangle^2  & =  \sum_{\alpha,i} |\partial_i u_\alpha|^2 - \sum_\alpha 2  k u_\alpha  \sum_i v_i\partial_i u_\alpha + \sum_\alpha k^2 \sum_i v_i^2 u_\alpha^2 =  \sum_{\alpha,i} |\partial_i u_\alpha|^2 - k^2 |u|^2 \\
 & = \sum_{\alpha} |\nabla_{\V}^{\mathrm{tot}} u_\alpha|^2 - k^2|u|^2 = \sum_{\alpha} |\nabla_{\V} u_\alpha|^2 = |\nabla_{\V}^{\E}u|^2.
 \end{split}
\]
Integrating on the sphere, we get
\[
\int_{S_x M} \sum_{\alpha,i} \langle \nabla_{\V} u_\alpha ,\e_i\rangle^2 =  \|\nabla_{\V}^{\E} u\|^2_{L^2(S_xM)} = \langle \Delta_{\V}^{\E} u,u \rangle = k(n+k-2)\|u\|^2.
\]
This completes the proof.
\end{proof}

We now deal with the term involving $\mc{G}^{\E}$ in \eqref{equation:splitting}:

\begin{lemma}
\label{lemma:inequality-2}
Let $u \in C^\infty(M,\Omega_k \otimes \E)$ such that $\imath_v u$ is of degree $k-1$. Then we have:
\[
 \langle \mc{G}^{\Lambda^p} u, \nabla_{\V}^{\Lambda^p}u \rangle = (n+2k-4) \|\imath_v u\|^2 + p\|u\|^2,
 \]
 and
 \[
  \langle \mc{G}^{\Sym^2} u, \nabla_{\V}^{\Sym^2}u \rangle = 2 (n+2k-4) \|\imath_v u\|^2 + 2\|u\|^2.
 \]
\end{lemma}

\begin{proof}
We treat the two cases separately. \\

\emph{Case $\E = \Lambda^p TM$:} We use the same conventions as in the proof of Lemma \ref{lemma:inequality-1}. Firstly, note that $\imath_v u$ being of degree $k - 1$ is equivalent to:
	\begin{equation}\label{eq:v_contraction}
		\imath_v u = \sum_\alpha u_\alpha. \imath_v \e_\alpha = \sum_{\alpha, i} u_\alpha v_i. \imath_{\e_i}\e_\alpha = (n + 2(k - 1))^{-1} \partial_i u_\alpha |v|^2 \imath_{\e_i} \e_\alpha	.
	\end{equation}	
	By \eqref{eq:F^E}, we may write $\mc{G}^{\Lambda^p} = \sum_i \e_i \otimes G^{\Lambda^p}(v, \e_i)$, so that on $S_xM$:
	\begin{align*}
		 \langle \mc{G}^{\Lambda^p} u, \nabla_{\V}^{\Lambda^p}u \rangle_{L^2} &= \sum_{\alpha, \beta} \int_{S_xM} u_\alpha. \langle{G^{\Lambda^p}(v, \nabla_{\V}u_\beta) \e_{\alpha}, \e_{\beta}}\rangle\\
		 &= \sum_{\alpha, \beta, i, j} \int_{S_xM} u_\alpha. \langle{\dotso\wedge \langle{G(v, \nabla_{\V}u_\beta) \e_{\alpha_i}, \e_j}\rangle.\e_j \wedge\dotso, \e_\beta}\rangle\\
		 &= \sum_{\alpha, \beta, i, j} \int_{S_xM} u_\alpha. G(v, \nabla_{\V}u_\beta, \e_{\alpha_i}, \e_j).\langle{\e_j \wedge \imath_{\e_{\alpha_i}} \e_{\alpha}, \e_\beta}\rangle\\
		 &= \sum_{\alpha, \beta, i, j} \int_{S_xM} u_\alpha. \Big(\langle{v, \e_{i}}\rangle.\langle{\nabla_{\V} u_\beta, \e_{j}}\rangle - \langle{v, \e_{j}}\rangle.\langle{\nabla_{\V} u_\beta, \e_{i}}\rangle\Big). \langle{\imath_{\e_{i}} \e_{\alpha}, \imath_{\e_j} \e_\beta}\rangle,
	\end{align*}
	where we used that the wedge product is adjoint to contraction. Denote by $A$ and $B$ the first and second terms in the last expression, respectively. We compute:
	\begin{align*}
		A &= \sum_{\alpha, \beta, i, j} \int_{S_xM} u_\alpha. \langle{v, \e_{i}}\rangle.\langle{\nabla_{\V} u_\beta, \e_{j}}\rangle. \langle{\imath_{\e_{i}} \e_{\alpha}, \imath_{\e_j} \e_\beta}\rangle = \sum_{\beta, i} \int_{S_xM} (\partial_j u_\beta - k.u_\beta v_j). \langle{\imath_{v} u, \imath_{\e_j} \e_\beta}\rangle\\
		&= (n + k - 2) \int_{S_xM} |\imath_v u|^2,
	\end{align*}
	where we used \eqref{eq:vertical_deriv_i} in the first line and \eqref{eq:v_contraction} in the second one. Next, for $B$ we have:
	\begin{align*}
		B &= \sum_{\alpha, \beta, i, j} \int_{S_xM} u_\alpha. \langle{v, \e_{j}}\rangle.(\partial_i u_\beta - k.u_\beta v_i). \langle{\imath_{\e_{i}} \e_{\alpha}, \imath_{\e_j} \e_\beta}\rangle\\
		&= -k \int_{S_xM} |\imath_v u|^2 +  \sum_{\alpha, \beta, i, j} \int_{S_xM} u_\alpha. (\partial_i(u_\beta v_j) - u_\beta. \delta_{ij}) \langle{\imath_{\e_{i}} \e_{\alpha}, \imath_{\e_j} \e_\beta}\rangle\\
		&= -k \int_{S_xM} |\imath_v u|^2 - p\int_{S_xM} |u|^2 + (n + 2(k - 1))^{-1} \sum_{\alpha, \beta, i, j} \int_{S_xM} u_\alpha.\partial_i(\langle{\imath_{\e_{i}} \e_{\alpha}, \partial_j u_\beta. \imath_{\e_j}\e_{\beta}}\rangle)\\
		&= -k \int_{S_xM} |\imath_v u|^2 - p\int_{S_xM} |u|^2 + (n + 2(k - 1))^{-1}\sum_{\alpha, \beta, i, j}\int_{S_xM} u_\alpha\\
		&\times (2v_i. \partial_j u_\beta + \partial_i \partial_j u_\beta).\langle{\imath_{\e_{i}} \e_{\alpha}, \imath_{\e_j}\e_{\beta}}\rangle = -(k - 2) \int_{S_xM} |\imath_v u|^2 - p\int_{S_xM} |u|^2,
	\end{align*}
	where we used \eqref{eq:vertical_deriv_i} in the first line, \eqref{eq:v_contraction} in the third and final lines, as well as the fact that $u_\alpha$'s are of degree $k$ while $\partial_i \partial_j u_\alpha$ is of strictly smaller degree. We also used the following identity on $p$-forms: $\sum_i \e_i \wedge \imath_{\e_i} = p. \id$. This completes the proof when $\E =\Lambda^p TM$.\\
	
	\emph{Case $\E =\Sym^2 TM$:} Write $u = \sum_{i, j} u_{ij} \e_{ij}$ where $\e_{ij} = \e_i^* \otimes \e_j$ and $u_{ij} = u_{ji}$ by symmetry of $u$. We begin by observing that $u(v)$ is of degree $k - 1$ translates into:
	\begin{equation}\label{eq:u(v)relation}
		u(v) = \sum_{i,j} u_{ij} v_i \e_j = (n + 2(k - 1))^{-1} \sum_{i, j}\partial_i u_{ij} \e_j.
	\end{equation}
	Therefore
	\small
	\begin{align*}
			\langle{\mc{G}^{\Sym^2}u, \nabla_{\V}^{\Sym^2}u}\rangle &= \sum_{i, \ell, m} \int \langle{\e_i, \nabla_{\V}u_{\ell m}}\rangle. \langle{G^{\Sym^2}(v, \e_i)u, \e_{\ell m}}\rangle = \sum_{i, j, \ell, m} \int u_{ij} \langle{[G(v, \nabla_{\V} u_{\ell m}), \e_{ij}], \e_{\ell m}}\rangle\\
		&= \sum_{i, j, \ell, m} \int u_{ij} \big\langle{\e_i^* \otimes G(v, \nabla_{\V}u_{\ell m})\e_j + (G(v, \nabla_{\V}u_{\ell m}) \e_i)^* \otimes \e_j, \e_{\ell m}}\big\rangle\\
		&= \sum_{i, j, \ell, m} \int u_{ij} \big(\delta_{i\ell} \langle{G(v, \nabla_{\V}u_{\ell m})\e_j, \e_{m}}\rangle + \delta_{jm}\langle{G(v, \nabla_{\V}u_{\ell m}) \e_i, \e_{\ell}}\rangle\big)\\
		&\hspace{-1.5cm}= 2\sum_{i, j, \ell} \int u_{ij} \big\langle{G(v, \nabla_{\V}u_{i\ell})\e_j, \e_{\ell}}\big\rangle = 2\sum_{i, j, \ell} \int u_{ij} \Big(\langle{v, \e_j}\rangle.\langle{\nabla_{\V}u_{i\ell}, \e_{\ell}}\rangle - \langle{v, \e_{\ell}}\rangle.\langle{\nabla_{\V}u_{i\ell}, \e_j}\rangle\Big).
	\end{align*}
	\normalsize
	Note that we used the symmetry $u_{ij} = u_{ji}$ in the last line. Denoting by $A$ the first and by $B$ the second term, we get using \eqref{eq:u(v)relation}:
	\begin{align*}
		A = \sum_{i, j, \ell} \int u_{ij}. v_j. (\partial_{\ell} u_{i\ell} - k u_{i\ell}v_{\ell}) = -k \|u(v)\|^2 +  (n + 2k - 2) \|u(v)\|^2 = (n + k - 2) \|\imath_v u\|^2.
	\end{align*}
	For the term $B$, we have:
	\begin{align*}
		B &= \sum_{i, j, \ell} \int u_{ij}. v_{\ell}. (\partial_j u_{i\ell} - k.v_j. u_{i\ell}) = -k\|\imath_vu\|^2 + \sum_{i, j, \ell} \int u_{ij} (\partial_j (v_{\ell} u_{i\ell}) - \delta_{j\ell} u_{i\ell})\\
		&=-k\|\imath_vu\|^2 - \|u\|^2 + (n + 2(k - 1))^{-1} \sum_{i, j, \ell} \int u_{ij} \partial_j (|v|^2 \partial_{\ell} u_{i\ell})\\
		&=-k\|\imath_vu\|^2 - \|u\|^2 + 2(n + 2(k - 1))^{-1} \sum_{i, j, \ell} \int u_{ij} v_j \partial_{\ell} u_{i\ell} = -(k - 2) \|\imath_vu\|^2 - \|u\|^2.
	\end{align*}
	Here we used \eqref{eq:u(v)relation} in the second and last lines, and the fact that $\partial_j \partial_{\ell} u_{i\ell}$ is of degree at most $k - 2$, while $u_{ij}$ is of degree $k$. This completes the proof.
\end{proof}

\begin{lemma}
\label{lemma:inequality-3}
Given $u \in C^\infty(M,\Omega_k \otimes \E)$
\[
\langle R\nabla_{\V}^{\E}u, \nabla_{\V}^{\E}u \rangle \leq - \delta k(n+k-2) \|u\|^2 .
\]
\end{lemma}

\begin{proof}
This is immediate using the upper bound on the sectional curvature:
\[
\langle R\nabla_{\V}^{\E}u, \nabla_{\V}^{\E}u \rangle \leq -\delta \|\nabla_{\V} u\|^2 = -\delta \langle \Delta_{\V} u, u \rangle = -\delta k(n+k-2) \|u\|^2.
\]
\end{proof}

Overall, the right-hand side of \eqref{equation:central-equality} can be bounded using Lemmas \ref{lemma:inequality-1}, \ref{lemma:inequality-2}, \ref{lemma:inequality-3}. We derive from the three previous lemmas the following lower bound:

\begin{lemma}
\label{lemma:lower-bound-xplus}
Let $u \in C^\infty(M,\Omega_k \otimes \Lambda^p TM)$ such that $\imath_v u$ is of degree $k-1$, and $k \geq 1$. Then:
\[
\|\X_+ u\|^2 \geq \dfrac{k+1}{k(n+2k)} \left( B^{\Lambda^p}_{n,k,\delta} \|u\|^2 -\frac{1+\delta}{2}(n+2k-4)\|\imath_v u\|^2 \right),
\]
where $B^{\Lambda^p}_{n,k,\delta}$ is defined in \eqref{equation:constant-b}.
\end{lemma}

\begin{proof}
We apply the localized Pestov identity \eqref{equation:local-pestov}. The terms $\|\X_-u\|^2$ and $\|Z(u)\|^2$ are simply bounded from below by $\geq 0$. Using Lemmas \ref{lemma:inequality-1}, \ref{lemma:inequality-2}, \ref{lemma:inequality-3}, we obtain:
\[
\begin{split}
\dfrac{k(n+2k)}{k+1}\|\X_+ u\|^2 & \geq - \langle R\nabla_{\V}^{\Lambda^p}u, \nabla_{\V}^{\Lambda^p}u \rangle -  \langle \mc{F}^{\Lambda^p}u, \nabla_{\V}^{\E}u \rangle  \\
& \geq \left(  \delta k(n+k-2) -  \dfrac{(1+\delta)p}{2}  -  \dfrac{2p}{3}(1-\delta) \left[k(n+k-2)(n-1)\right]^{1/2} \right) \|u\|^2 \\
& \hspace{3cm} - \frac{1+\delta}{2}(n+2k-4)\|\imath_v u\|^2 \\
&  = B^{\Lambda^p}_{n,k,\delta}\|u\|^2 - \frac{1+\delta}{2}(n+2k-4)\|\imath_v u\|^2.
\end{split}
\]
\end{proof}

We now have:

\begin{lemma}
\label{lemma:inequality-4}
Under the assumptions \eqref{equation:assumptions}, and if $k \geq 2$, we have:
\[
\|\X_-u\|^2 \geq \dfrac{k}{(k-1)(n+2k-2)} \left(B^{\Lambda^{p-1}}_{n,k-1,\delta} \|\imath_v u\|^2 -\frac{1+\delta}{2}(n+2k-6)\|\imath_v \imath_v u\|^2 \right),
\]
with the convention that $p=2$ if $\E=\Sym^2 TM$.
\end{lemma}

\begin{proof}
Observe that $\X \imath_v u =\imath_v \X u = \imath_v \X_- u$, and thus applying Lemma \ref{lemma:lower-bound-xplus} with $\imath_v u$, we get: 
\[
\begin{split}
\|\X_- u\|^2 & \geq \|\imath_v \X_- u\|^2 \\
&  = \|\X \imath_v u\|^2 = \|\X_- \imath_v u\|^2 + \|\X_+ \imath_v u\|^2 \\
&  \geq \dfrac{k}{(k-1)(n+2k-2)} \left(B_{n,k-1,\delta}^{\Lambda^{p-1}} \|\imath_v u\|^2  -\frac{1+\delta}{2}(n+2k-6)\|\imath_v \imath_v u\|^2 \right).
\end{split}
\]
\end{proof}

We can now complete the proof of Proposition \ref{proposition:calcul}. Inserting the bounds of Lemmas \ref{lemma:inequality-1}, \ref{lemma:inequality-2}, \ref{lemma:inequality-3}, \ref{lemma:inequality-4} in \eqref{equation:central-equality}, we obtain:
\begin{equation}
\label{equation:total}
\begin{split}
& \dfrac{(n+k-2)(n+2k-4)}{n+k-3} \dfrac{k}{(k-1)(n+2k-2)} \left(B^{\Lambda^{p-1}}_{n,k-1,\delta} \|\imath_v u\|^2 - \frac{1+\delta}{2}(n+2k-6)\|\imath_v \imath_v u\|^2\right) \\
& \leq \text{RHS of \eqref{equation:central-equality}} \\
& \leq  \dfrac{1+\delta}{2} \left(\eps(\E)(n+2k-4) \|\imath_v u\|^2 + p\|u\|^2 \right) +  \dfrac{2p}{3}(1-\delta) \left[k(n+k-2)(n-1)\right]^{1/2} \|u\|^2 \\
& \hspace{3cm} - \delta k(n+k-2)\|u\|^2,
\end{split}
\end{equation}
where $\eps(\E)=2$ if $\E=\Sym^2 TM$ and $\eps(\E)=1$ if $\E=\Lambda^p TM$. This inequality can now be rearranged as
\[
B^{\E}_{n,k,\delta}\|u\|^2 + C^{\E}_{n,k,\delta}\|\imath_v u\|^2 - D_{n,k,\delta}\|\imath_v \imath_v u\|^2 \leq 0,
\]
\end{proof}

\begin{remark}\rm
	It can be easily checked that the Cauchy-Schwarz estimate in \eqref{eq:Cauchy-Schwarz} is \emph{sharp for $k = 1$, $p = 1$}. We believe that for degrees $k > 1$ this estimate is not sharp. Let us assume $p = 1$ and $k = 3$ in what follows. Let $C(n) > 0$ be the optimal constant such that $F(u) := \sum_{i = 1}^n \int_{S_xM} |u - u_i \e_i|.|\nabla_{\V} u_i|.\omega_i(v) \leq C(n) \|u\|^2$ holds for $u = \sum_{i = 1}^n u_i \e_i$, where $u_i$ are spherical harmonics of degree $k = 3$ on $\mathbb{S}^{n-1}$; here $\omega_i(v) \in [\frac{1}{2}, 1]$ are certain weights that are obtained by collecting leftover terms in the proof of Lemma \ref{lemma:bk} (these are not necessary in the article). By our estimate \eqref{eq:Cauchy-Schwarz} we get $C(n) \leq \sqrt{3(n^2 - 1)}$. The space of spherical harmonics of degree $k = 3$ for $n = 4, 6, 8$ has dimension $16, 50, 112$, respectively. Results (up to three decimal digits) of a computer program evaluating $F(u)$ at random spherical harmonics in these cases are given in the following table. Here $\delta_{\Lambda^1, \mathrm{new}}(n)$ denotes the new corresponding value of $\delta_{\Lambda^1}(n)$ in Theorem \ref{theorem:invariant-structures}. Results also indicate that the quotient $\frac{C(n)}{\sqrt{3(n^2 - 1)}}$ converges to $1$ as $n \to \infty$, i.e. that the Cauchy-Schwarz bound is asymptotically optimal, and that this particular estimate does not suffice to prove Conjecture \ref{conjecture:brin} directly (however an optimal estimate would improve Theorem \ref{theorem:ergodicity}).
	
	\[
\begin{tabular}{ | c | c | c | c | }
 \hline			
    $n$& $4$ & $6$ & $8$\\
   \hline
   $C(n)$ & $5.294$ & $8.614$ & $12.193$ \\
   \hline
 $\frac{C(n)}{\sqrt{3(n^2-1)}}$ & $0.789$ & $0.841$ & $0.886$\\
 \hline  
 $\delta_{\Lambda^1, \mathrm{new}}(n)$ & $0.267$ & $0.262$ & $0.261$\\
 \hline  
  \# random points & $500$ & $50$ & $5$\\
 \hline  
 \end{tabular}
 \]
\end{remark}

\subsection{Non-existence of invariant structures}

We can now deduce Theorem \ref{theorem:invariant-structures} from the previous paragraph.

\begin{proof}[Proof of Theorem \ref{theorem:invariant-structures}]
We treat separately the $\Sym^2 TM$ and the $\Lambda^p TM$ cases. \\

\emph{Case $\E=\Lambda^p TM$:} We consider an element $f \in C^\infty(SM,\Lambda^p \mc{N})$ such that $\X f = 0$. By \cite{Guillarmou-Paternain-Salo-Uhlmann-16}, it has finite degree, so we can write $f = f_0 + \ldots + f_k$, where $f_i \in C^\infty(M,\Omega_i \otimes \Lambda^p TM)$ and $f_k \neq 0$. We set $u := f_k$. Then $u$ is a normal conformal Killing tensor in the sense of Definition \ref{definition:nckt} and it thus satisfies the conclusions of Proposition \ref{proposition:calcul}. Moreover, $\imath_v \imath_v u = 0$. We then argue as follows: if $B^{\Lambda^p}_{n,k,\delta} > 0$ and $B^{\Lambda^p}_{n,k,\delta} + C^{\Lambda^p}_{n,k,\delta} > 0$, then $u= 0$. Indeed, assume that this condition is satisfied. Then, if $C^{\Lambda^p}_{n,k,\delta} \geq 0$, we obtain $B^{\Lambda^p}_{n,k,\delta} \|u\|^2 \leq 0$, hence $u = 0$. If $C^{\Lambda^p}_{n,k,\delta} \leq 0$, then we simply use the bound $\|\imath_vu\|^2 \leq \|u\|^2$ which gives $(B^{\Lambda^p}_{n,k,\delta} + C^{\Lambda^p}_{n,k,\delta})\|u\|^2 \leq 0$ and thus $u=0$. We shall now see that this condition translates into a pinching condition on $\delta$. \\

We introduce the notation $r_{n,p,k}:=\frac{2p}{3} \sqrt{\frac{n-1}{k(n+k-2)}}$ and $s_{n,p,k}:=(n+2k-2)r_{n,p,k}$. Then one can write
\tiny
\[
\begin{split}B^{\Lambda^p}_{n,k,\delta}&=\delta\left(k(n+k-2)-\frac p2+k(n+k-2) r_{n,p,k} \right)-\left(\frac p2+k(n+k-2)r_{n,p,k} \right)\\
&=k(n+k-2)\left(\delta\left(1-\frac{p}{2k(n+k-2)}+r_{n,p,k} \right)-\left(\frac{p}{2k(n+k-2)}+r_{n,p,k} \right)\right),
\end{split}
\]
\normalsize
and
\tiny
\[
\begin{split}C^{\Lambda^p}_{n,k,\delta}&=\dfrac{k(n+k-2)(n+2k-4)}{(k-1)(n+k-3)(n+2k-2)} B_{n,k-1,\delta}^{\Lambda^{p-1}} - \frac{ (n+2k-4)(1+\delta)}{2} \\
&=\frac{k(n+k-2)(n+2k-4)}{n+2k-2}\left[\delta\left(1-\frac{(p-1)}{2(k-1)(n+k-3)}+r_{n,p-1,k-1}-\frac{(n+2k-2)}{2k(n+k-2)} \right)\right.\\
&-\left.\left(\frac{(p-1)}{2(n+k-3)(k-1)}+r_{n,p-1,k-1}+\frac{ (n+2k-2)}{2k(n+k-2)} \right)\right].
\end{split}
\]
\normalsize
It follows that $B^{\Lambda^p}_{n,k,\delta}>0$ if and only if 
\be\label{delta1}\delta>\delta_1:=\frac{\frac{p}{2k(n+k-2)}+r_{n,p,k}}{1-\frac{p}{2k(n+k-2)}+r_{n,p,k} }=\frac{\frac{p(n+2k-2)}{2k(n+k-2)}+s_{n,p,k}}{n+2k-2-\frac{p(n+2k-2)}{2k(n+k-2)}+s_{n,p,k} }\ee
and  $B^{\Lambda^p}_{n,k,\delta}+C^{\Lambda^p}_{n,k,\delta}>0$ if and only if 
\tiny
\[
\begin{split}\delta>\delta_2&:=\frac{\frac{(p-1)}{2(n+k-3)(k-1)}+r_{n,p-1,k-1}+\frac{(n+2k-2)}{2k(n+k-2)}+\frac{n+2k-2}{n+2k-4}\Big(\frac{p}{2k(n+k-2)}+r_{n,p,k}\Big)}{1-\frac{(p-1)}{2(n+k-3)(k-1)}+r_{n,p-1,k-1}-\frac{ (n+2k-2)}{2k(n+k-2)}+\frac{n+2k-2}{n+2k-4}\Big(1-\frac{p}{2k(n+k-2)}+r_{n,p,k}\Big) }\\
&=\frac{\frac{(p-1)(n+2k-4)}{2(n+k-3)(k-1)}+s_{n,p-1,k-1}+\frac{ (n+2k-2)(n+2k-4)}{2k(n+k-2)}+\frac{p(n+2k-2)}{2k(n+k-2)}+s_{n,p,k}}{n+2k-4-\frac{(p-1)(n+2k-4)}{2(n+k-3)(k-1)}+s_{n,p-1,k-1}-\frac{ (n+2k-2)(n+2k-4)}{2k(n+k-2)}+{n+2k-2}-\frac{p(n+2k-2)}{2k(n+k-2)}+s_{n,p,k} }.\end{split}
\]
\normalsize
In other words, we get:
\[
\left(B^{\Lambda^p}_{n,k,\delta} > 0 \text{ and } B^{\Lambda^p}_{n,k,\delta} + C^{\Lambda^p}_{n,k,\delta} > 0\right) \Leftrightarrow \delta > \max(\delta_1,\delta_2).
\] 
We further claim that the following holds:

\begin{lemma}
\label{lemma:growth}
The functions $(k,p) \mapsto \delta_1(n,k,p), \delta_2(n,k,p)$ are increasing in $p$ and decreasing in $k$, for $k\geq 2$, $p\geq 1$.
\end{lemma}

\begin{proof}
Note that 
$$s_{n,p,k}=\frac{2p\sqrt{n-1}}{3} \sqrt{\frac{(n+2k-2)^2}{k(n+k-2)}}=\frac{2p\sqrt{n-1}}{3} \sqrt{4+\frac{(n-2)^2}{k(n+k-2)}}$$
is decreasing in $k$ and increasing in $p$, and by \eqref{delta1} one can write 
$$\frac{1}{\delta_1}-1=\frac{n+2k-2-\frac{p(n+2k-2)}{k(n+k-2)}}{\frac{p(n+2k-2)}{2k(n+k-2)}+s_{n,p,k}}=\frac{n+2k-2-\frac{p}{n+k-2}-\frac{p}{k}}{\frac{1}{2}\Big(\frac{p}{n+k-2}+\frac{p}{k}\Big)+s_{n,p,k}},
$$
thus showing that $\delta_1$ is decreasing in $k$, increasing in $p$.

We claim that $\delta_2$ is also decreasing in $k$. For that, we introduce
\[
F_{n,k,p} := \frac{(p-1)(n+2k-4)}{2(n+k-3)(k-1)}+\frac{(n+2k-2)(n+2k-4)}{2k(n+k-2)}+\frac{p(n+2k-2)}{2k(n+k-2)}.
\]
From the above expression for $\delta_2$, we get 
$$\frac1{\delta_2}-1=\frac{2n+4k-6-2 F_{n,k,p}}{s_{n,p-1,k-1}+s_{n,p,k}+F_{n,k,p}}$$
so our claim would follow if $F_{n,k,p}$ is decreasing in $k$ and increasing in $p$. The statement for $p$ is immediate. As to the statement for $k$, this indeed holds for $p\ge 2$ since one can write the above expression as 
$$2F_{n,k,p}=\frac{p-1}{k-1}+\frac{p-1}{n+k-3}+4+\frac{ (n-2)^2}{k(n+k-2)}+\frac{p-2}{k}+\frac{p-2}{n+k-2},
$$
which is sum of decreasing functions. For $p=1$ this argument fails, but in this case one can write
$$\frac1{\delta_2}-1=\frac{2n+4k-6-\frac{ (n+2k-2)(n+2k-3)}{k(n+k-2)}}{\frac{ (n+2k-2)(n+2k-3)}{2k(n+k-2)}+s_{n,1,k}}=\frac{2-\frac1k-\frac1{n+k-2}}{\frac1{2k}+\frac1{2(n+k-2)}+\frac{s_{n,1,k}}{n+2k-3}},$$
and the numerator is increasing in $k$, whereas the denominator is decreasing in $k$.
\end{proof}

We then set
\begin{align*}
\delta_{\Lambda^1}(n) &:= \max\big(\delta_1(n,k=3,p=1),\delta_2(n,k=3,p=1)\big),\\
\delta_{\mathrm{U}(3)}(7) &:= \max\big(\delta_1(n=7,k=3,p=2),\delta_2(n=7,k=3,p=2)\big) = 0.4962...,\\
 \delta_{\mathrm{G}_2}(8) &:= \max\big(\delta_1(n=8,k=3,p=3),\delta_2(n=8,k=3,p=3)\big)=0.6212...,\\
 \delta_{\mathrm{E}_7}(134) &:= \max\big(\delta_1(n = 134,k=3,p=3),\delta_2(n = 134,k=3,p=3)\big) = 0.5788...
 \end{align*}
We deal separately with different cases. The non-zero invariant section $f$ has constant norm and, up to rescaling, we can always assume that it is equal to $1$.

\begin{enumerate}[itemsep=5pt]

\item We first consider the case where $p=1$ and $f \in C^\infty(SM,\mc{N})$ is flow-invariant and odd. Assuming that the manifold $(M^n,g)$ is $\delta$-pinched and $\delta > \delta_{\Lambda^1}(n)$, we then obtain that $f_k = 0$ if $k \geq 3$ since Lemma \ref{lemma:growth} ensures that $\delta > \delta_{\Lambda^1}(n) > \max\big(\delta_1(n,k,p=1),\delta_2(n,k,p=1)\big)$ for all $k\geq 3$. The invariant section $f \in C^\infty(SM, \mc{N})$ is thus of degree $\leq 2$, hence of pure degree $1$ since it is odd. As a consequence, $f = \pi_1^* f_1$ for some $f_1 \in C^\infty(M,TM \otimes TM)$. This implies the existence of $J \in C^\infty(M,\End(TM))$ such that $(f_1)_x(v) = J(x)v$ for all $x \in M,v \in T_xM$. Moreover, the properties $f_1(v) \in v^\bot$ and $|f_1(v)|^2 = 1$ implies that $J$ is skew-symmetric and orthogonal, hence an almost-complex structure. The flow-invariance property $\X f =0$ then translates into $(\nabla_v J) v = 0$, that is $J$ endows $(M^n,g)$ with a nearly-Kähler structure \cite{Gray-76, Nagy-02}. By \cite[Proposition 2.1]{Nagy-02} this implies that the universal cover $(\widetilde{M},\widetilde{g})$ splits as a product of a Kähler manifold and a strictly nearly-Kähler manifold (i.e. such that $\nabla_v J \neq 0$ whenever $v \neq 0$). Since $(\widetilde{M},\widetilde{g})$ has negative sectional curvature, it cannot be a Riemannian product, so one of the factors is trivial. Hence: either $(M,g,J)$ is strictly nearly-Kähler, in which case $g$ has positive scalar curvature by \cite[Theorem 1.1]{Nagy-02} and this is impossible, or it is Kähler, in which case the pinching satisfies $\delta \leq 0.25$ by \cite{Berger-60-1}\footnote{Berger \cite{Berger-60-1} covers the positive sectional curvature case but the proof is verbatim the same in the negatively curved case.} and this is also a contradiction. 

\item We now deal with the exceptional case where $n=7$ and $f \in C^\infty(SM,\Lambda^2 \mc{N})$ is an odd flow-invariant almost-complex structure. If $\delta > \delta_{\mathrm{U}(3)}(7)$, we conclude that $f$ has pure degree $1$ since it is odd. As before, we obtain that $f = \pi_1^* f_1$ for some $f_1 \in C^\infty(M,\Lambda^3 TM)$ with the property that for all $v$, $f_1(v,\cdot,\cdot)$ is an almost-complex structure on $v^\bot$. This implies that $f_1$ is a $\mathrm{G}_2$-structure on $M$. The flow-invariance condition $\X f = 0$ then translates into $(\nabla_v f_1)(v,\cdot,\cdot)=0$, that is $f_1$ is nearly-parallel (see \cite[Theorem 5.2, second line]{Fernandez-Gray-82} or \cite[Proposition 2.4, case (6)]{Alexandrov-Semmelmann-12} for definitions). By \cite[Proposition 3.10]{Friedrich-Kath-Moroianu-Semmelmann-97}, nearly-parallel $\mathrm{G}_2$-manifolds are Einstein with non-negative scalar curvature and this contradicts the negative sectional curvature of $g$.

\item Consider now the case $n=8$ and $f \in C^\infty(SM,\Lambda^3 \mc{N})$ is an odd flow-invariant $\mathrm{G}_2$-structure on $SM$. If $\delta > \delta_{\mathrm{G}_2}(8)$, the same argument as before shows that $f = \pi_1^*f_1$, for some $f_1 \in C^\infty(M,\Lambda^4 TM)$ with the property that for every $v$, $f_1(v,\cdot,\cdot,\cdot) = 0$ is a $\mathrm{G}_2$-structure on $v^\bot$, that is $f_1$ is a $\mathrm{Spin}(7)$-structure. Moreover, the flow-invariance condition $\X f =0$ translates into $(\nabla_v f_1)(v,\cdot,\cdot,\cdot)$, that is $f_1$ is nearly-parallel. By \cite[Lemma 4.4]{Fernandez-86}, nearly-parallel $\mathrm{Spin}(7)$-structures are necessarily parallel and this implies that $g$ is Ricci-flat \cite{Bonan-66}, which is a contradiction.

\item Eventually, we deal with the case where $n=134$ and $f \in C^\infty(SM,\Lambda^3 \mc{N})$ is an odd flow-invariant Lie bracket of degree $k \geq 3$ (by Lemma \ref{e7}). Taking $u := f_k \neq 0$, we get that $u$ is a normal twisted conformal Killing tensor of degree $k\geq 3$. If $\delta > \delta_{\mathrm{E}_7}(134)$, then we get as before that $u=0$ which is a contradiction. \\

\end{enumerate}

\emph{Case $\E=\Sym^2 TM$:} We treat separately the case $k=2$ and $k \geq 4$. First of all, let us assume $f \in C^\infty(SM,\Sym^2 \mc{N})$ is a flow-invariant even projector with (even) degree $k \geq 4$. Let $u:=f_k \neq 0$. Then $u$ is a normal twisted conformal Killing tensor satisfying the conclusions of Proposition \ref{proposition:calcul}. We use the bound $\|\imath_v \imath_v u\|^2 \leq \|\imath_v u\|^2$, which gives, setting $\overline{C}^{\Sym^2}_{n,k,\delta} := C^{\Sym^2}_{n,k,\delta} - D_{n,k,\delta}$, that:
\[
B^{\Sym^2}_{n,k,\delta}\|u\|^2 + \overline{C}^{\Sym^2}_{n,k,\delta} \|\imath_vu\|^2 \leq 0.
\]
As before, if $B^{\Sym^2}_{n,k,\delta} > 0$ and $B^{\Sym^2}_{n,k,\delta} + \overline{C}^{\Sym^2}_{n,k,\delta} > 0$, then we can conclude that $u=0$. Now, we have:
\[
\left(B^{\Sym^2}_{n,k,\delta} > 0 \text{ and } B^{\Sym^2}_{n,k,\delta} +\overline{C}^{\Sym^2}_{n,k,\delta} > 0\right) \Leftrightarrow \delta > \max(\delta_1,\delta'_2),
\] 
where $\delta_1$ is the same as before \eqref{delta1} (with $p=2$) and
\[
\begin{split}\delta'_2&:=\frac{\frac{1}{2(k-1)(n+k-3)}+r_{n,1,k-1}+\frac{ (n+2k-2)}{k(n+k-2)}+\frac{n+2k-2}{n+2k-4}\big(\frac{1}{k(n+k-2)}+r_{n,2,k}\big)+\frac{n+2k-6}{2(k-1)(n+k-3)}}{1-\frac{1}{2(k-1)(n+k-3)}+r_{n,1,k-1}-\frac{ (n+2k-2)}{k(n+k-2)}+\frac{n+2k-2}{n+2k-4}\big(1-\frac{1}{k(n+k-2)}+r_{n,2,k}\big)-\frac{n+2k-6}{2(k-1)(n+k-3)} }\\
&=\frac{\frac{(n+2k-4)(n+2k-5)}{2(k-1)(n+k-3)}+s_{n,1,k-1}+\frac{(n+2k-2)(n+2k-3)}{k(n+k-2)}+s_{n,2,k}}{2n+4k-6-\frac{(n+2k-4)(n+2k-5)}{2(k-1)(n+k-3)}+s_{n,1,k-1}-\frac{ (n+2k-2)(n+2k-3)}{k(n+k-2)}+s_{n,2,k} },\end{split}
\]
is defined for $k \geq 4$. We claim that the following holds:

\begin{lemma}
\label{lemma:growth2}
For $n\geq 7$, the function $k \mapsto \delta'_2(n,k)$ is decreasing in $k$.
\end{lemma}

\begin{proof}
We have 
\[
\begin{split}
\frac1{\delta'_2}-1&=\frac{2n+4k-6-\frac{(n+2k-4)(n+2k-5)}{(k-1)(n+k-3)}-\frac{2(n+2k-2)(n+2k-3)}{k(n+k-2)}}{\frac{(n+2k-4)(n+2k-5)}{2(k-1)(n+k-3)}+s_{n,1,k-1}+\frac{(n+2k-2)(n+2k-3)}{k(n+k-2)}+s_{n,2,k}}\\
&=
\frac{2-\frac{(n+2k-4)(n+2k-5)}{(k-1)(n+k-3)(n+2k-3)}-\frac2k-\frac2{n+k-2}}{\frac{(n+2k-4)(n+2k-5)}{2(k-1)(n+k-3)(n+2k-3)}+\frac{s_{n,1,k-1}}{n+2k-3}+\frac{(n+2k-2)}{k(n+k-2)}+\frac{s_{n,2,k}}{n+2k-3}}.
\end{split}
\]
The numerator is increasing in $k$ and the denominator is decreasing in $k$ since
for $n\geq 7$ the expression
\[\begin{split}\frac{(n+2k-4)(n+2k-5)}{(k-1)(n+k-3)(n+2k-3)}&=\frac1{2(k-1)}+\frac{n-5}{2(k-1)(n+2k-3)}\\
&+\frac{3}{2(n+k-3)} + \frac{n - 7}{2(n+k-3)(n + 2k - 3)}\end{split}
\]
is decreasing in $k$.
\end{proof}

As a consequence, we deduce that for $\delta > \max\big(\delta_1(n,k=4),\delta_2'(n,k=4)\big)$, the flow-invariant even orthogonal projection $f$ has degree $\leq 2$, so it now remains to study the case of degree $2$. \\

In this case, by Lemma \ref{lemma:algebra}, we know that $f = \frac{r}{n} \mathbbm{1}_{TM} + f_2$, with $u= f_2$. Since $\X \mathbbm{1}_{TM} = 0$, we get that $\X_{\pm} u = 0$. Following the proof of Proposition \ref{proposition:calcul} (the $\X_-u$ term disappears from the computation) we obtain similarly as in \eqref{equation:total} that:
\begin{equation}
\label{equation:total2}
\begin{split}
0 \leq \dfrac{1+\delta}{2} \left(2(n+2\cdot 2-4) \|\imath_v u\|^2 + 2\|u\|^2 \right)& +  \dfrac{2\cdot 2}{3}(1-\delta) \left[2(n+2-2)(n-1)\right]^{1/2} \|u\|^2 \\
&- \delta \cdot 2(n+2-2)\|u\|^2.
\end{split}
\end{equation}
Observe that $\imath_v f = 0 = \frac{r}{n} v + \imath_v u$. Moreover, at a given point $x_0 \in M$, we have
\[
\begin{split}
\|f\|^2_{L^2(S_{x_0}M)} & = \int_{S_{x_0}M} \Tr(f^2) dv = r \vol(\Ss^{n-1}) \\
& = \|f_0\|^2_{L^2(S_{x_0}M)} + \|u\|^2_{L^2(S_{x_0}M)} = \frac{r^2}{n} \vol(\Ss^{n-1}) + \|u\|^2_{L^2(S_{x_0}M)},
\end{split}
\]
and we obtain 
\[
\begin{array}{l}
\|u\|^2_{L^2(S_{x_0}M)} = r \big(1-\frac{r}{n}\big)\vol(\Ss^{n-1}), \\
\|\imath_v u\|^2_{L^2(S_{x_0}M)} = \frac{r^2}{n^2}\vol(\Ss^{n-1}) = \frac{r}{n(n-r)}\|u\|^2_{L^2(S_{x_0}M)}.
\end{array}
\]
Plugging the previous equality in \eqref{equation:total2}, we then obtain:
\[
\left(2 \delta n - \dfrac{4}{3}(1-\delta)[2n(n-1)]^{1/2} - (1+\delta) \dfrac{n}{n-r} \right)\|u\|^2 \leq 0.
\]
The term in the brackets is positive if and only if
\[
\delta > \delta'_2(n,k=2,r) := \dfrac{\frac{4}{3}[2n(n-1)]^{1/2}+ \frac{n}{n-r}}{2n+\frac{4}{3}[2n(n-1)]^{1/2}- \frac{n}{n-r}}.
\]
This term is clearly increasing in $r$ and $r \leq \min\big(\rho(n)-1,\frac{n-2}{2}\big)$.
As a consequence, we deduce that for
\[
\delta > \max\left(\delta'_2\left(n,k=2,r=\min\left(\rho(n)-1,\frac{n-2}{2}\right)\right),\delta'_2(n,k=4)\right) = \delta'_2(n,k=4),
\]
the flow-invariant even orthogonal projector $f$ has degree $0$, which contradicts Lemma \ref{lemma:algebra}.
\end{proof}

\bibliographystyle{alpha}

\bibliography{Biblio}

\begin{thebibliography}{{Bri}75b}

\bibitem[ABF11]{Agricola-Becker-Bender-Friedrich-11}
Ilka {Agricola}, Julia {Becker-Bender}, and Thomas {Friedrich}.
\newblock On the topology and the geometry of {${\rm SO}(3)$}-manifolds.
\newblock {\em Ann. Global Anal. Geom.}, 40(1):67--84, 2011.

\bibitem[{Ada}62]{Adams-62}
John~Frank {Adams}.
\newblock Vector fields on spheres.
\newblock {\em Ann. of Math. (2)}, 75:603--632, 1962.

\bibitem[AS12]{Alexandrov-Semmelmann-12}
Bogdan {Alexandrov} and Uwe {Semmelmann}.
\newblock Deformations of nearly parallel {${\rm G}_2$}-structures.
\newblock {\em Asian J. Math.}, 16(4):713--744, 2012.

\bibitem[Ber60]{Berger-60-1}
Marcel Berger.
\newblock Pincement riemannien et pincement holomorphe.
\newblock {\em Ann. Scuola Norm. Sup. Pisa Cl. Sci. (3)}, 14:151--159, 1960.

\bibitem[{Bes}08]{Besse}
Arthur {Besse}.
\newblock {\em {Einstein manifolds}}.
\newblock Berlin: Springer, 2008.

\bibitem[BG80]{Brin-Gromov-80}
Michael {Brin} and Mikhael {Gromov}.
\newblock On the ergodicity of frame flows.
\newblock {\em Invent. Math.}, 60(1):1--7, 1980.

\bibitem[BK78]{Bourguignon-Karcher-78}
Jean-Pierre {Bourguignon} and Hermann {Karcher}.
\newblock Curvature operators: pinching estimates and geometric examples.
\newblock {\em Ann. Sci. \'{E}cole Norm. Sup. (4)}, 11(1):71--92, 1978.

\bibitem[BK84]{Brin-Karcher-83}
Michael {Brin} and Hermann {Karcher}.
\newblock Frame flows on manifolds with pinched negative curvature.
\newblock {\em Compositio Math.}, 52(3):275--297, 1984.

\bibitem[{Bon}66]{Bonan-66}
Edmond {Bonan}.
\newblock Sur des vari\'{e}t\'{e}s riemanniennes \`a groupe d'holonomie
  {$G_{2}$} ou spin {$(7)$}.
\newblock {\em C. R. Acad. Sci. Paris S\'{e}r. A-B}, 262:A127--A129, 1966.

\bibitem[BP03]{Burns-Pollicott-03}
Keith {Burns} and Mark {Pollicott}.
\newblock Stable ergodicity and frame flows.
\newblock {\em Geom. Dedicata}, 98:189--210, 2003.

\bibitem[{Bri}75a]{Brin-75-2}
Michael {Brin}.
\newblock Topological transitivity of a certain class of dynamical systems, and
  flows of frames on manifolds of negative curvature.
\newblock {\em Funkcional. Anal. i Prilo\v{z}en.}, 9(1):9--19, 1975.

\bibitem[{Bri}75b]{Brin-75-1}
Michael {Brin}.
\newblock The topology of group extensions of {$C$}-systems.
\newblock {\em Mat. Zametki}, 18(3):453--465, 1975.

\bibitem[{Bri}82]{Brin-82}
Michael {Brin}.
\newblock Ergodic theory of frame flows.
\newblock In {\em Ergodic theory and dynamical systems, {II} ({C}ollege {P}ark,
  {M}d., 1979/1980)}, volume~21 of {\em Progr. Math.}, pages 163--183.
  Birkh\"{a}user, Boston, Mass., 1982.

\bibitem[{Bry}87]{Bryant1987}
Robert {Bryant}.
\newblock Metrics with exceptional holonomy.
\newblock {\em Ann. of Math.}, 126(3):525--576, 1987.

\bibitem[{\v{C}}C06]{Cadek-Crabb-06}
Martin {\v{C}}adek and Michael Crabb.
\newblock {$G$}-structures on spheres.
\newblock {\em Proc. London Math. Soc. (3)}, 93(3):791--816, 2006.

\bibitem[CL21]{Cekic-Lefeuvre-20}
Mihajlo {Ceki{\'c}} and Thibault {Lefeuvre}.
\newblock {Generic Dynamical Properties of Connections on Vector Bundles}.
\newblock {\em International Mathematics Research Notices}, 04 2021.
\newblock rnab069.

\bibitem[CL22]{Cekic-Lefeuvre-21-1}
Mihajlo {Ceki{\'c}} and Thibault {Lefeuvre}.
\newblock {The Holonomy Inverse Problem}.
\newblock {\em to appear in J. Eur. Math. Soc.}, 2022.

\bibitem[CL23]{Cekic-Lefeuvre-22}
Mihajlo {Ceki{\'c}} and Thibault {Lefeuvre}.
\newblock {Isospectral connections, ergodicity of frame flows, and polynomial
  maps between spheres}.
\newblock {\em to appear in Ann. Sci. Éc. Norm. Supér.}, 2023.

\bibitem[CL24]{Cekic-Lefeuvre-24}
Mihajlo {Ceki{\'c}} and Thibault {Lefeuvre}.
\newblock {Semiclassical analysis on principal bundles}.
\newblock {\em arXiv:2405.14846}, 2024.

\bibitem[CLMS22]{Cekic-Lefeuvre-Moroianu-Semmelmann-22}
Mihajlo Ceki\'c, Thibault Lefeuvre, Andrei Moroianu, and Uwe Semmelmann.
\newblock Towards {B}rin's conjecture on frame flow ergodicity: new progress
  and perspectives.
\newblock {\em Math. Res. Rep.}, 3:21--34, 2022.

\bibitem[CLMS24]{Cekic-Lefeuvre-Moroianu-Semmelmann-24}
Mihajlo Ceki\'c, Thibault Lefeuvre, Andrei Moroianu, and Uwe Semmelmann.
\newblock On the ergodicity of unitary frame flows on {K}\"ahler manifolds.
\newblock {\em Ergodic Theory Dynam. Systems}, 44(8):2143--2172, 2024.

\bibitem[CS98]{Croke-Sharafutdinov-98}
Christopher~B. Croke and Vladimir~A. Sharafutdinov.
\newblock Spectral rigidity of a compact negatively curved manifold.
\newblock {\em Topology}, 37(6):1265--1273, 1998.

\bibitem[dC92]{Do-Carmo-book}
Manfredo Perdig\~{a}o do~Carmo.
\newblock {\em Riemannian geometry}.
\newblock Mathematics: Theory \& Applications. Birkh\"{a}user Boston, Inc.,
  Boston, MA, 1992.
\newblock Translated from the second Portuguese edition by Francis Flaherty.

\bibitem[DK01]{Davis-Kirk-01}
James~F. Davis and Paul Kirk.
\newblock {\em Lecture notes in algebraic topology}, volume~35 of {\em Graduate
  Studies in Mathematics}.
\newblock American Mathematical Society, Providence, RI, 2001.

\bibitem[DKR24]{Dolgopyat-Kanigowski-Rodriguez-Hertz-24}
Dmitry {Dolgopyat}, Adam {Kanigowski}, and Federico {Rodríguez-Hertz}.
\newblock Exponential mixing implies {B}ernoulli.
\newblock {\em Ann. of Math. (2)}, 199(3):1225--1292, 2024.

\bibitem[DS10]{Dairbekov-Sharafutdinov-10}
Nurlan~S. Dairbekov and Vladimir~A. Sharafutdinov.
\newblock Conformal {K}illing symmetric tensor fields on {R}iemannian
  manifolds.
\newblock {\em Mat. Tr.}, 13(1):85--145, 2010.

\bibitem[Fer86]{Fernandez-86}
Marisa Fern\'{a}ndez.
\newblock A classification of {R}iemannian manifolds with structure group
  {${\rm Spin}(7)$}.
\newblock {\em Ann. Mat. Pura Appl. (4)}, 143:101--122, 1986.

\bibitem[FG82]{Fernandez-Gray-82}
Marisa Fern\'{a}ndez and Alfred Gray.
\newblock Riemannian manifolds with structure group {$G_{2}$}.
\newblock {\em Ann. Mat. Pura Appl. (4)}, 132:19--45 (1983), 1982.

\bibitem[FKMS97]{Friedrich-Kath-Moroianu-Semmelmann-97}
Thomas Friedrich, Ines Kath, Andrei Moroianu, and Uwe Semmelmann.
\newblock On nearly parallel {$G_2$}-structures.
\newblock {\em J. Geom. Phys.}, 23(3-4):259--286, 1997.

\bibitem[GBL23]{Bonthonneau-Lefeuvre-21}
Yannick Guedes~Bonthonneau and Thibault Lefeuvre.
\newblock Radial source estimates in {H}\"older-{Z}ygmund spaces for hyperbolic
  dynamics.
\newblock {\em Ann. H. Lebesgue}, 6:643--686, 2023.

\bibitem[GK80a]{Guillemin-Kazhdan-80}
Victor Guillemin and David Kazhdan.
\newblock Some inverse spectral results for negatively curved {$2$}-manifolds.
\newblock {\em Topology}, 19(3):301--312, 1980.

\bibitem[GK80b]{Guillemin-Kazhdan-80-2}
Victor Guillemin and David Kazhdan.
\newblock Some inverse spectral results for negatively curved {$n$}-manifolds.
\newblock In {\em Geometry of the {L}aplace operator ({P}roc. {S}ympos. {P}ure
  {M}ath., {U}niv. {H}awaii, {H}onolulu, {H}awaii, 1979)}, Proc. Sympos. Pure
  Math., XXXVI, pages 153--180. Amer. Math. Soc., Providence, R.I., 1980.

\bibitem[GPSU16]{Guillarmou-Paternain-Salo-Uhlmann-16}
Colin {Guillarmou}, Gabriel~P. {Paternain}, Mikko {Salo}, and Gunther
  {Uhlmann}.
\newblock The {X}-ray transform for connections in negative curvature.
\newblock {\em Comm. Math. Phys.}, 343(1):83--127, 2016.

\bibitem[Gra76]{Gray-76}
Alfred Gray.
\newblock The structure of nearly {K}\"{a}hler manifolds.
\newblock {\em Math. Ann.}, 223(3):233--248, 1976.

\bibitem[Hat02]{Hatcher-02}
Allen Hatcher.
\newblock {\em Algebraic topology}.
\newblock Cambridge University Press, Cambridge, 2002.

\bibitem[Hel01]{Helgason-01}
Sigurdur Helgason.
\newblock {\em Differential geometry, {L}ie groups, and symmetric spaces},
  volume~34 of {\em Graduate Studies in Mathematics}.
\newblock American Mathematical Society, Providence, RI, 2001.
\newblock Corrected reprint of the 1978 original.

\bibitem[HMS16]{Heil-Moroianu-Semmelmann-16}
Konstantin Heil, Andrei Moroianu, and Uwe Semmelmann.
\newblock Killing and conformal {K}illing tensors.
\newblock {\em J. Geom. Phys.}, 106:383--400, 2016.

\bibitem[HP06]{Hasselblatt-Pesin-06}
Boris Hasselblatt and Yakov Pesin.
\newblock Partially hyperbolic dynamical systems.
\newblock In {\em Handbook of dynamical systems. {V}ol. 1{B}}, pages 1--55.
  Elsevier B. V., Amsterdam, 2006.

\bibitem[Kra66]{Kraines}
Vivian~Yoh Kraines.
\newblock Topology of quaternionic manifolds.
\newblock {\em Trans. Amer. Math. Soc.}, 122(2):357--367, 1966.

\bibitem[Lef23]{Lefeuvre-21}
Thibault Lefeuvre.
\newblock Isometric extensions of {A}nosov flows via microlocal analysis.
\newblock {\em Comm. Math. Phys.}, 399(1):453--479, 2023.

\bibitem[Leo71]{Leonard-71}
Peter Leonard.
\newblock {$G$}-structures on spheres.
\newblock {\em Trans. Amer. Math. Soc.}, 157:311--327, 1971.

\bibitem[MPR90]{McKay-Patera-Rand-90}
William~G. McKay, Jiri Patera, and David~W. Rand.
\newblock {\em Tables of representations of simple Lie algebras. Vol. I.
  Exceptional simple Lie algebras}.
\newblock Universit\'{e} de Montr\'{e}al, Centre de Recherches Mathématiques,
  Montreal, QC, 190.

\bibitem[Nag02]{Nagy-02}
Paul-Andi Nagy.
\newblock On nearly-{K}\"{a}hler geometry.
\newblock {\em Ann. Global Anal. Geom.}, 22(2):167--178, 2002.

\bibitem[Pat99]{Paternain-99}
Gabriel~P. Paternain.
\newblock {\em Geodesic flows}, volume 180 of {\em Progress in Mathematics}.
\newblock Birkh\"{a}user Boston, Inc., Boston, MA, 1999.

\bibitem[PSU15]{Paternain-Salo-Uhlmann-15}
Gabriel~P. Paternain, Mikko Salo, and Gunther Uhlmann.
\newblock Invariant distributions, {B}eurling transforms and tensor tomography
  in higher dimensions.
\newblock {\em Math. Ann.}, 363(1-2):305--362, 2015.

\bibitem[PZ24]{Pollicott-Zhang-24}
Mark {Pollicott} and Daofei {Zhang}.
\newblock {Rapid mixing for compact group extensions of hyperbolic flows}.
\newblock {\em arXiv:2405.07356}, 2024.

\end{thebibliography}

\end{document}